\documentclass{article}%
\usepackage[latin9]{inputenc}
\usepackage{amsmath}
\usepackage{amsfonts}
\usepackage{amssymb}
\usepackage{amsthm}
\usepackage{graphicx}%
\providecommand{\U}[1]{\protect\rule{.1in}{.1in}}
\newtheorem{theorem}{Theorem}[section]

\newtheorem{lemma}[theorem]{Lemma}
\newtheorem{proposition}[theorem]{Proposition}
\newtheorem{remark}[theorem]{Remark}

\theoremstyle{definition}
\newtheorem{definition}[theorem]{Definition}

\theoremstyle{remark}
\newtheorem{example}[theorem]{Example}

\numberwithin{equation}{section}

\DeclareMathOperator{\sech}{sech}
\begin{document}

\title{Representation of solutions to the one-dimensional Schrödinger equation in terms of Neumann series of Bessel functions}
\author{Vladislav V. Kravchenko$^{1}$, Luis J. Navarro$^{2}$, Sergii M. Torba$^{1}$\\$^{1}${\footnotesize Department of Mathematics, CINVESTAV del IPN, Unidad
Quer\'{e}taro }\\{\footnotesize Libramiento Norponiente \# 2000 Fracc. Real de Juriquilla
\ Quer\'{e}taro, Qro., CP 76230, M\'{e}xico}\\$^{2}${\footnotesize Department of Pure and Applied Mathematics, Simon Bolivar
University, Caracas 1080-A, Venezuela}\\{\footnotesize vkravchenko@math.cinvestav.edu.mx}
\\{\footnotesize ljnavarro@usb.ve}
\\{\footnotesize storba@math.cinvestav.edu.mx}}
\maketitle

\begin{abstract}
A new representation of solutions to the equation $-y^{\prime\prime
}+q(x)y=\omega^{2}y$ is obtained. For every $x$ the solution is represented as
a Neumann series of Bessel functions depending on the spectral parameter
$\omega$. Due to the fact that the representation is obtained using the
corresponding transmutation operator, a partial sum of the series approximates
the solution uniformly with respect to $\omega$ which makes it especially
convenient for the approximate solution of spectral problems. The numerical
method based on the proposed approach allows one to compute large sets of
eigendata with a nondeteriorating accuracy.

\end{abstract}

\footnotetext[3]{Research was supported by CONACYT via the projects 166141 and
222478, Mexico.
\par
During this work Luis Navarro was with the CINVESTAV supported by the
Secretary of Foreign Affairs of Mexico postdoctoral fellowship.}

\section{Introduction}

We consider the equation%
\begin{equation}
-y^{\prime\prime}+q(x)y=\omega^{2}y \label{SL_omega}%
\end{equation}
on a finite interval $(0,b)$. We assume $q$ being a continuous complex valued
function of an independent real variable $x\in\lbrack 0,b]$ and $\omega$ an
arbitrary complex number.

The main result of the paper is a new representation for solutions of
(\ref{SL_omega}) and for their derivatives in the form of Neumann series of
Bessel functions with explicit formulas for the coefficients. We obtain that
two linearly independent solutions of (\ref{SL_omega}) have the form%

\begin{equation}
c(\omega,x)=\cos\omega x+2\sum_{n=0}^{\infty}(-1)^{n}\beta_{2n}(x)j_{2n}%
(\omega x)\label{Intro c}%
\end{equation}
and
\begin{equation}
s(\omega,x)=\sin\omega x+2\sum_{n=0}^{\infty}(-1)^{n}\beta_{2n+1}%
(x)j_{2n+1}(\omega x)\label{Intro s}%
\end{equation}
where $j_{k}$ stands for the spherical Bessel function of order $k$ and the
functions $\beta_{k}$ are calculated following a relatively simple recursive
integration procedure. The series are uniformly convergent both with respect
to $x$ and to $\omega$. The representations are obtained with the aid of the
transmutation (transformation) operators related with (\ref{SL_omega}) (for
the theory of such operators we refer to \cite{BegehrGilbert}, \cite{Carroll},
\cite{LevitanInverse}, \cite{Marchenko}, \cite{Sitnik}, \cite{Trimeche}). Due
to this fact and since the kernel of a transmutation operator, realized in the
form of a Volterra integral operator, is independent of the spectral
parameter, it is not difficult to prove that the partial sums of the series
obtained approximate the solutions uniformly with respect to $\omega$ (Theorem
\ref{Th Representation of solutions via Bessel} and Remark
\ref{Rem Estimates for derivatives}). This makes them especially attractive
for approximate solving of spectral problems.

The possibility to dispose of explicit formulas (\ref{beta direct definition})
for the coefficients $\beta_{k}$ in the series comes from a mapping property
of the transmutation operators discovered in \cite{CKT}, see also \cite{KT
Transmut} and \cite{KT AnalyticApprox}. Due to that property, in spite of not
knowing the kernel of the transmutation operator, one can however construct
the set of images under the action of the transmutation operator of all
nonnegative integer powers of the independent variable $x$. Using this we
obtain a Fourier-Legendre expansion of the transmutation kernel and use it to
write down the new representations for solutions of (\ref{SL_omega}).
Estimates for the rate of convergence of the Fourier-Legendre series are
obtained in dependence on the smoothness of $q$.

The efficiency and the uniform accuracy of the solution representations is
illustrated by some numerical examples which show that in several seconds one
can compute hundreds or if necessary thousands of eigendata with essentially
the same and small enough absolute error.

The one-dimensional Schr\"{o}dinger equation (\ref{SL_omega}) is, of course,
one of the most fundamental and classical objects of study in the theory of
differentials equations and mathematical physics. Its applications are
uncountable. Any new result for this equation and especially a new
representation of its solutions can lead to unforeseen applications. Since the
new representation possesses such a unique and, in fact, amazing feature of
uniformity with respect to the spectral parameter $\omega$, its clear and
immediate application is to approximate solution of spectral and scattering
problems. We explore this use of our main result and show that without using
any elaborate numerical technique and simply programming the analytical
formulas obtained in the present work one can compute huge amounts of
eigendata with a uniform accuracy guaranteed, very fast, and in general for
complex valued coefficients. At present no other algorithm offers similar
possibilities, and we emphasize that a new numerical method is only one
possible application of the representation obtained. We expect that the new
representation will be used for obtaining asymptotic relations and solving
inverse problems. Moreover, a crucial role is played by the transmutation
operator. This is another fundamental object of the theory of differential
equations and especially of the theory of inverse spectral and scattering
problems. Its basic properties are well understood, but the difficulties with
construction of its integral kernel have always restricted its practical use.
In this relation we mention the paper \cite{Boumenir2006} where analytic
approximation formulas for the integral kernel were obtained and the recent
publications \cite{KT Transmut}, \cite{KT AnalyticApprox} where another
procedure of analytical approximation was proposed. To the difference of those
previous results, in the present work we obtain an exact formula for the
transmutation operator, its kernel is represented in the form of a
Fourier-Legendre series with explicit formulas for the coefficients. The use
of the transmutation operator is not limited to Sturm-Liouville equations. In
particular, it is applied to relate partial differential equations (see, e.g.,
\cite{BegehrGilbert}) and hence to solve problems involving PDEs with variable
coefficients. The representation of the transmutation kernel proposed here has
a convenient structure for this sort of applications. For example, it may be
used to construct complete systems of solutions of PDEs related via the
transmutation to those with known complete systems of solutions (see, e.g.,
\cite{Colton1976}).

Finally, let us notice that the Neumann series of Bessel functions represent
another classical notion of mathematical analysis. They were first studied by
the German mathematician
Carl Gottfried Neumann in 1867 and are named after him. The theory
of Neumann series was developed later by
L. B. Gegenbauer in 1877. We refer to another important paper on this subject
\cite{Wilkins}. For more recent results we refer to \cite{Baricz2012},
\cite{Pogany2009} and references therein. In an interesting research reported
in \cite{Chebli1994} and \cite{Fitouhi1990} there appears a representation of
solutions of Sturm-Liouville equations in the form of Neumann series of Bessel
functions different to the representation obtained in the present work. The
representation from \cite{Chebli1994} and \cite{Fitouhi1990} does not possess
the uniformity with respect to $\omega$ to the difference from our
representation, and the convergence of the series which is guaranteed on a
certain interval of $x$ for holomorphic $q$ only is achieved due to the
exponential decay of $j_{n}(z)$ when $n\rightarrow\infty$. Apart from that
previous work, to our best knowledge, the Neumann series of Bessel functions
have not been used to represent solutions of a general linear differential
equation. The attractive features of the representation presented here
indicate that the Neumann series of Bessel functions should be considered as a
natural and important object of study in the theory of linear differential equations.

The paper is structured as follows. In Section
\ref{Sect Transmutations and Powers} we introduce some necessary notations,
definitions and properties concerning special systems of functions related to
(\ref{SL_omega}) and called formal powers, as well as the transmutation
operators. In Section \ref{Sect Fourier-Legendre for the Kernel} we show that
the kernel of a transmutation operator admits the representation $K(x,t)=\sum_{k=0}^{\infty}
\frac{\beta_{k}(x)}{x}P_{k}\left(  \frac{t}{x}\right)  $ where $P_{k}$ are
Legendre polynomials and the coefficients $\beta_{k}$ are defined with the aid
of the formal powers. We prove a direct and an inverse results on the rate of
convergence of the series in dependence on the smoothness of $q$. In Section
\ref{Sect Representation Solutions} the Fourier-Legendre expansion of the
transmutation kernel is used to obtain the main result of this work, the
representations (\ref{Intro c}) and (\ref{Intro s}), and to prove that the
partial sums of these series give us a uniform approximation of the solutions
with respect to the spectral parameter $\omega$. In Section
\ref{Sect Representation Derivatives} we obtain analogous results for the
derivatives of the solutions $c(\omega,x)$ and $s(\omega,x)$. Again we prove
the uniform approximation with respect to $\omega$. Since the coefficients
$\beta_{k}$ (and their counterparts $\gamma_{k}$ appearing in the
representations of the derivatives) are the main ingredient of the
representations which depends on $q$, besides their direct definition in terms
of formal powers (\ref{beta direct definition}) it is desirable to dispose of
a most efficient and stable  procedure for their numerical computation. In
Section \ref{Sect Sequence Equations} we propose\ one such procedure which
proved to work numerically much better than (\ref{beta direct definition}) and
converted the representations (\ref{Intro c}) and (\ref{Intro s}) into a
powerful numerical method for solving initial value and spectral problems for
(\ref{SL_omega}). In Section \ref{Sect Num Exp} we confirm this affirmation
with several numerical experiments.

\section{Transmutations and formal
powers\label{Sect Transmutations and Powers}}

The definition of the transmutation operator as given in this paper requires the potential $q$ to be defined on the symmetric interval $[-b,b]$ (see \cite{Marchenko}, \cite{KT Transmut}). However as we explain later in this section, the results of the present work do not depend on the continuation of the potential onto negative values of $x$. For that reason throughout this section we assume that equation \eqref{SL_omega} is defined on the symmetric segment $[-b,b]$ and that the potential $q$ is continuous on this segment and in the rest of the paper only the segment $[0,b]$ is considered.

Throughout the paper we suppose that $f$ is a non-vanishing solution (in general, complex-valued) of the
equation
\begin{equation}
f^{\prime\prime}-qf=0\label{SLhom}%
\end{equation}
on $[0,b]$ ($[-b,b]$ for this section) such that
\[
f(0)=1.
\]
The existence of such
solution\footnote{In fact the only reason for the requirement of the absence
of zeros of the function $f$ is to make sure that the auxiliary functions
(\ref{phik}) and (\ref{psik}) be well defined. As was shown in
\cite{KrTNewSPPS} this can be done even without such requirement, but
corresponding formulas are relatively more complicated.} for any
complex-valued $q\in C[-b,b]$ was proved in \cite[Remark 5]{KrPorter2010} (see
also \cite{Camporesi et al 2011}). Denote $h:=f^{\prime}(0)$.

Consider two sequences of recursive integrals (see \cite{KrCV08}, \cite{KMoT},
\cite{KrPorter2010})
\begin{equation*}
X^{(0)}(x)\equiv1,\qquad X^{(n)}(x)=n\int_{0}^{x}X^{(n-1)}(s)\left(
f^{2}(s)\right)  ^{(-1)^{n}}\,\mathrm{d}s,\qquad n=1,2,\ldots\label{Xn}%
\end{equation*}
and
\begin{equation*}
\widetilde{X}^{(0)}\equiv1,\qquad\widetilde{X}^{(n)}(x)=n\int_{0}%
^{x}\widetilde{X}^{(n-1)}(s)\left(  f^{2}(s)\right)  ^{(-1)^{n-1}}%
\,\mathrm{d}s,\qquad n=1,2,\ldots. \label{Xtiln}%
\end{equation*}

\begin{definition}
\label{Def Formal powers phik and psik}The families of functions $\left\{
\varphi_{k}\right\}  _{k=0}^{\infty}$ and $\left\{  \psi_{k}\right\}
_{k=0}^{\infty}$ constructed according to the rules
\begin{equation}
\varphi_{k}(x)=%
\begin{cases}
f(x)X^{(k)}(x), & k\text{\ odd},\\
f(x)\widetilde{X}^{(k)}(x), & k\text{\ even}%
\end{cases}
\label{phik}%
\end{equation}
and
\begin{equation}
\psi_{k}(x)=%
\begin{cases}
\dfrac{\widetilde{X}^{(k)}(x)}{f(x)}, & k\text{\ odd,}\\
\dfrac{X^{(k)}(x)}{f(x)}, & k\text{\ even}.
\end{cases}
\label{psik}%
\end{equation}
are called the systems of formal powers associated with $f$.
\end{definition}

\begin{remark}
The formal powers arise in the spectral parameter power series (SPPS)
representation for solutions of \eqref{SL_omega} (see \cite{KKRosu},
\cite{KrCV08}, \cite{KMoT}, \cite{KrPorter2010}).
\end{remark}

\begin{theorem}
\label{Th Transmutation} Let $q\in C[-b,b]$. Then there exists a unique
complex valued function $K(x,t)\in C^{1}([-b,b]\times\lbrack-b,b])$ such that
the Volterra integral operator
\begin{equation}
Tu(x)=u(x)+\int_{-x}^{x}K(x,t)u(t)dt\label{Tmain}%
\end{equation}
defined on $C[-b,b]$ satisfies the equality
\[
\left(  -\frac{d^{2}}{dx^{2}}+q(x)\right)  T[u]=T\left[  -\frac{d^{2}}{dx^{2}%
}(u)\right]
\]
for any $u\in C^{2}[-b,b]$ and
\[
T[1]=f\text{.}%
\]

\end{theorem}

For the proof of this fact we refer to \cite[Theorem 3.1.1]{Marchenko52}. Slightly different proof and references to earlier publications are given in \cite{KT Transmut}.

$T$ maps any solution $v$ of the equation $v^{\prime\prime}+\omega^{2}v=0$
into a solution $y$ of equation (\ref{SL_omega}) with the following
correspondence of the initial values $y(0)=v(0)$, $y^{\prime}(0)=v^{\prime
}(0)+hv(0)$.

In particular, we introduce two linearly independent solutions of
(\ref{SL_omega}),%
\begin{equation}
c(\omega,x):=T\left[  \cos\omega x\right]  \quad\text{and}\quad s(\omega
,x):=T\left[  \sin\omega x\right]  . \label{c and s definition}%
\end{equation}

Note that the definition of the transmutation operator \eqref{Tmain} requires
knowledge of the integral kernel $K$ only in the regions $R_{1}:=\{0\leq x\leq
b,\ |t|\leq x\}$ and $R_{2}:=\{-b\leq x\leq0,\ |t|\leq|x|\}$. Moreover, these
two regions are independent in the following sense. The integral kernel $K$ in
$R_{1}$ depends on the values of the potential $q$ only on $[0,b]$ and does
not depend on values for $x<0$ (see \eqref{IntEq for K}), and in $R_{2}$
depends on the values on $[-b,0]$ and does not depend on the values for $x>0$. The value of  $T[u](x)$ for $x\geq0$ does not require the knowledge of $K$ on $R_{2}$. The same happens to the formal powers $\{\varphi_{k}\}$ and
$\{\psi_{k}\}$ whose values on $[0,b]$ are independent on the potential $q$ on $[-b,0)$. Therefore from now on we restrict the presentation to the segment $[0,b]$, all the results for the segment $[-b,0]$ are similar. One of the advantages of restricting the consideration to the segment $[0,b]$ consists in the knowledge of the initial values of the solutions $c(\omega,x)$ and $s(\omega,x)$ in the origin which is convenient for solving initial value and spectral problems on $[0,b]$.

The following mapping property plays a crucial role in what follows.

\begin{proposition}
[\cite{CKT}]\label{Prop Mapping property}
\begin{equation}
T\left[  x^{k}\right]  =\varphi_{k}(x)\qquad\text{for any}\ k\in\mathbb{N}%
\cup\left\{  0\right\}  . \label{mapping property}%
\end{equation}
\end{proposition}

Thus, even without knowing the transmutation kernel $K(x,t)$ it is possible to make use of the transmutation operator $T$ because the images of all nonnegative integer powers of $x$ can be calculated following Definition \ref{Def Formal powers phik and psik}. This result is used in the next section for obtaining an exact representation for the kernel $K(x,t)$ in the form of a Fourier-Legendre series (Theorem \ref{Th transmutation kernel via beta and P}).

\section{The Fourier-Legendre expansion of the transmutation
kernel\label{Sect Fourier-Legendre for the Kernel}}

Let $P_{n}$ denote the Legendre polynomial of order $n$, $l_{k,n}$ be the
corresponding coefficient of $x^{k}$, that is $P_{n}(x)=\sum_{k=0}^{n}%
l_{k,n}x^{k}$. Denote
\[
p_{j,k}:=\int_{-1}^{1}P_{j}\left(  y\right)  y^{k}dy,\qquad j,k\in
\mathbb{N}\cup\left\{  0\right\}  .
\]
Notice that $p_{j,k}=0$ when the parities of $j$ and $k$ do not coincide or
when $k<j$. For any $n\geq m$ and $\delta\in\{0,1\}$ (see, e.g., \cite{Prudnikov}),
\begin{equation*}
p_{2m+\delta,2n+\delta}  =\frac{\sqrt{\pi}\Gamma(2n+1+\delta)}{2^{2n+\delta}\Gamma(n-m+1)\Gamma
(\frac{3}{2}+n+m+\delta)}.
\end{equation*}

\begin{definition}
\label{Def beta}Let us introduce the following infinite system of functions
$\beta_{k}$, $k=0,1,\ldots$ defined recursively as follows
\[
\beta_{0}(x)=\frac{f(x)-1}{2}\text{,\qquad}\beta_{1}(x)=\frac{3}{2}\left(
\frac{\varphi_{1}(x)}{x}-1\right)  \text{,}%
\]
for any even $k>0$,%
\[
\beta_{k}(x)=\frac{1}{p_{kk}}\biggl(  \frac{\varphi_{k}(x)}{x^{k}}-1-
\sum_{\text{even }j=0}^{k-2} p_{jk}\beta_{j}(x)\biggr)
\]
and for any odd $k>1$,%
\[
\beta_{k}(x)=\frac{1}{p_{kk}}\biggl(  \frac{\varphi_{k}(x)}{x^{k}}-1-
\sum_{\text{odd }j=1}^{k-2} p_{jk}\beta_{j}(x)\biggr)  .
\]

\end{definition}

Below we show that the functions $\beta_{k}$ admit the following direct
definition as well%
\begin{equation}
\beta_{n}(x)=\frac{2n+1}{2}\biggl( \sum_{k=0}^{n}\frac{l_{k,n}\varphi_{k}%
(x)}{x^{k}}-1\biggr) . \label{beta direct definition}%
\end{equation}

\begin{theorem}
\label{Th transmutation kernel via beta and P}The transmutation kernel
$K(x,t)$ from Theorem \ref{Th Transmutation} has the form%
\begin{equation}
K(x,t)=%
{\displaystyle\sum\limits_{j=0}^{\infty}}
\frac{\beta_{j}(x)}{x}P_{j}\left(  \frac{t}{x}\right)  \label{K via beta}%
\end{equation}
where for every $x\in(0,b]$ the series converges uniformly with respect to
$t\in\lbrack-x,x]$.
\end{theorem}

\begin{proof}
Since $K\in C^{1}([-b,b]\times\lbrack-b,b])$, for any $x\in( 0,b]$ it admits
(see, e.g., \cite{Suetin}) a uniformly convergent Fourier-Legendre series of
the form $\sum_{j=0}^{\infty} A_{j}(x)P_{j}\left(  \frac{t}{x}\right)  $ where
for convenience we consider $A_{j}(x)=\frac{\alpha_{j}(x)}{x}$. Substitution
of this series into (\ref{mapping property}) gives us the equality
\[
\varphi_{k}(x)=x^{k}+x^{k}\sum_{j=0}^{\infty} \frac{\alpha_{j}(x)}{x}\int
_{-x}^{x}P_{j}\left(  \frac{t}{x}\right)  \left(  \frac{t}{x}\right)  ^{k}dt
\]
for any $k=0,1,2,\ldots$. The change of the variable $y:=t/x$ leads to the
equality%
\[
\varphi_{k}(x) =x^{k}\biggl(  1+\sum_{j=0}^{\infty} \alpha_{j}(x)\int_{-1}%
^{1}P_{j}\left(  y\right)  y^{k}dy\biggr)
=x^{k}\biggl(  1+\sum_{j=0}^{k} p_{j,k}\alpha_{j}(x)\biggr)  .
\]
Solution of this system of equations leads to the conclusion $\alpha_{j}%
=\beta_{j}$ defined by the recursive formulas from Definition \ref{Def beta},
and hence to (\ref{K via beta}).

Moreover, multiplying (\ref{K via beta}) by $P_{n}\left(  \frac{t}{x}\right)
$ and integrating we obtain%
\begin{equation}
\label{int = beta}\int_{-x}^{x}K(x,t)P_{n}\left(  \frac{t}{x}\right)  dt
=\sum_{j=0}^{\infty} \frac{\beta_{j}(x)}{x}\int_{-x}^{x}P_{j}\left(  \frac
{t}{x}\right)  P_{n}\left(  \frac{t}{x}\right)  dt =\frac{2}{2n+1}\beta
_{n}(x).
\end{equation}
Hence%
\begin{align*}
\beta_{n}(x)  &  =\frac{2n+1}{2}\int_{-x}^{x}K(x,t)P_{n}\left(  \frac{t}%
{x}\right)  dt =\frac{2n+1}{2}\sum_{k=0}^{n}\int_{-x}^{x}K(x,t)l_{k,n}\left(
\frac{t}{x}\right)  ^{k}dt\\
&  =\frac{2n+1}{2}\sum_{k=0}^{n}\frac{l_{k,n}}{x^{k}}\int_{-x}^{x}%
K(x,t)t^{k}dt =\frac{2n+1}{2}\sum_{k=0}^{n}\frac{l_{k,n}}{x^{k}}\left(  T\left[
x^{k}\right]  -x^{k}\right)  .
\end{align*}
Using Proposition \ref{Prop Mapping property} we obtain
\[
\beta_{n}(x)=\frac{2n+1}{2}\sum_{k=0}^{n}\frac{l_{k,n}}{x^{k}}\left(
\varphi_{k}(x)-x^{k}\right)
\]
from where (\ref{beta direct definition}) follows due to the fact that
$P_{n}(1)=1$.
\end{proof}

In the next two theorems we establish a relation between the rate of
convergence of partial sums of the series \eqref{K via beta} and the
smoothness of the potential $q$.
Recall that for $q\in C^{(p)}[0,b]$ the integral kernel $K$ is $p+1$ times
continuously differentiable with respect to each variable \cite[\S 2]%
{Marchenko} justifying definition \eqref{dtKBound}.

We will denote the partial sum of the series \eqref{K via beta} by
\[
K_{N}
(x,t):=\sum_{j=0}^{N} \frac{\beta_{j}(x)}{x}P_{j}\left(  \frac{t}{x}\right).
\]

\begin{theorem}
\label{Th Convergence rate Direct} Suppose that $q\in C^{(p)}[0,b]$ and
define
\begin{equation}
\label{dtKBound}M:= \max_{0\le x\le b,\ |t|\le x}\bigl| \partial^{p+1}_{t}
K(x,t)\bigr|.
\end{equation}
Then for all $N>p$, $0<x\le b$, and $|t|\le x$
\begin{equation}
\label{EstimateK ParSum}\bigl| K(x,t) - K_N(x,t)  \bigr|\le\frac{c_{p} M x^{p+1}}{N^{p+1/2}},
\end{equation}
where the constant $c_{p}$ does not depend on $q$ and $N$.
\end{theorem}

\begin{proof}
The following estimate for the remainder of the Fourier-Legendre series for a
function $g\in C^{(p+1)}[-1,1]$ is known (see the proof of Theorem 4.10 from
\cite{Suetin} together with \cite[Theorem 5.2.1]{Timan}): for all $N>p+1$
\begin{equation}
\label{EstimateSuetin}\max_{[-1,1]}|g(x)-g_{N}(x)|\le\frac{c_{p+1}M_{g}%
}{N^{p+1/2}},
\end{equation}
where $g_{N}(x)=\sum_{k=0}^{N} a_{k} P_{k}(x)$ is a partial sum of the
Fourier-Legendre series of the function $g$, $M_{g}:=\max_{[-1,1]}%
|g^{(p+1)}(x)|$ and the constant $c_{p+1}$ does not depend on $g$ and $N$.

For each $x\in(0,b]$ let us consider a function $g(y):=K(x, xy)$, $y\in
[-1,1]$. Since the integral kernel $K$ is $p+1$ times continuously
differentiable with respect to the second variable, $g\in C^{(p+1)}[-1,1]$,
and $g^{(p+1)}(y) = x^{p+1}\left.  \partial_{t}^{p+1}K(x,t)\right|  _{t=xy}$.
Hence
\begin{equation}
\label{EstimateDerG}\max_{[-1,1]}|g^{(p+1)}(y)|\le x^{p+1}\max_{y\in
[-1,1]}\left|  \left.  \partial_{t}^{p+1}K(x,t)\right|  _{t=xy}\right|  \le
x^{p+1}M.
\end{equation}

Now \eqref{EstimateK ParSum} follows directly from \eqref{EstimateSuetin} and
\eqref{EstimateDerG} noting that $\sum_{j=0}^{N} \frac{\beta_{j}(x)}x
P_{j}(y)$ is a partial sum of the Fourier-Legendre series for the function
$g(y)$.
\end{proof}

\begin{remark}
Actually, the value of the integral kernel $K$ at some point $(x,t)$, $x>0$,
depends only on values of the potential $q$ on the segment $[0,x]$ and does
not depend on values of $q$ on $(x,b]$, which can be deduced, e.g., from
\eqref{IntEq for K}. Therefore the estimate \eqref{EstimateK ParSum} remains
valid if we change all entries of $b$ in Theorem
\ref{Th Convergence rate Direct} by some $b^{\prime}$ such that $x\leq
b^{\prime}\leq b$.
\end{remark}

The next theorem partially inverts the result of Theorem
\ref{Th Convergence rate Direct}. We need the following auxiliary result on the connection between the smoothness of the function $K(x,\cdot)$ (for a fixed $x$) and the smoothness of the potential $q$. Even though similar results are well known, we are not aware of a precise reference. In any case, the statement of the lemma can be easily verified using the integral equation satisfied by $K$ (see, e.g., \cite[Sect. 1.2]{LevitanInverse} and \cite[\S 2]{Marchenko})
\begin{equation}
\label{IntEq for K}K(x,t) = \frac{h}2 + \frac{1}2\int_{0}^{\frac{x+t}2}
q(s)\,ds + \int_{0}^{\frac{x+t}2} \int_{0}^{\frac{x-t}2}q(\alpha
+\beta)K(\alpha+\beta, \alpha-\beta)\,d\beta\,d\alpha.
\end{equation}

\begin{lemma}\label{Lemma diff K and q}
Let $x>0$ be fixed. Suppose that there
exists an integer $p\ge 1$ such that
$K(x,\cdot)\in
C^{(p)}[-x,x]$. Then $q\in C^{(p-1)}[0,x]$. If additionally $K(x,\cdot)\in
C^{(r)}(-x,x)$ for some $p<r\le 2p+1$, then $q\in C^{(r-1)}{(0,x)}$.
\end{lemma}

\begin{theorem}
\label{Th Convergence rate Inverse} Let $x>0$ be fixed. Suppose that there
exist an integer $p\ge2$ and the constants $c$ and $\varepsilon>0$ such that
for all $N>p+1$ and all $t$, $|t|\le x$,
\begin{equation}
\label{EstimateInverse}\bigl| K(x,t) - K_N(x,t)  \bigr|\le\frac{c}{N^{p+\varepsilon}}.
\end{equation}
Then $q\in C^{([p/2]-1)}[0,x]\cap C^{(p-1)}(0,x)$, where $[a]$ denotes the
largest integer less or equal to $a$.
\end{theorem}

\begin{proof}
The sum $\sum_{j=0}^{N}\frac{\beta_{j}(x)}{x}P_{j}\bigl(\frac{\cdot}{x}\bigr)$
is a polynomial of degree at most $N$ approximating the function $K(x,\cdot)$.
That is, the inequality \eqref{EstimateInverse} provides an upper bound for
the best uniform approximation of the function $K(x,\cdot)$ by degree $N$
polynomials. The application of the inverse approximation theorem (see, e.g.,
\cite[Theorem 31]{KKTT}) leads to the conclusion that $K(x,\cdot)\in
C^{(p)}(-x,x)\cap C^{[p/2]}[-x,x]$. Now the proof follows directly from Lemma \ref{Lemma diff K and q}.
\end{proof}


\section{Representation for solutions of the Schrödinger
equation\label{Sect Representation Solutions}}

\begin{theorem}
\label{Th Representation of solutions via Bessel}The solutions $c(\omega,x)$
and $s(\omega,x)$ of equation \eqref{SL_omega} admit the following
representations
\begin{equation}
\label{c(omega,x) via bessel}%
\begin{split}
c(\omega,x)  &  =\cos\omega x+\sqrt{\frac{2\pi}{\omega x}}\sum_{n=0}^{\infty}
(-1)^{n}\beta_{2n}(x)J_{2n+1/2}(\omega x)\\
&  =\cos\omega x+2\sum_{n=0}^{\infty} (-1)^{n}\beta_{2n}(x)j_{2n}(\omega x)
\end{split}
\end{equation}
and%
\begin{equation}
\label{s(omega,x) via bessel}%
\begin{split}
s(\omega,x)  &  =\sin\omega x+\sqrt{\frac{2\pi}{\omega x}}\sum_{n=0}^{\infty}
(-1)^{n}\beta_{2n+1}(x)J_{2n+3/2}(\omega x)\\
&  =\sin\omega x+2\sum_{n=0}^{\infty} (-1)^{n}\beta_{2n+1}(x)j_{2n+1}(\omega
x)
\end{split}
\end{equation}
where $j_{k}$ stands for the spherical Bessel function of order $k$, the
series converge uniformly with respect to $x$ on $[0,b]$ and converge
uniformly with respect to $\omega$ on any compact subset of the complex plane
of the variable $\omega$. Moreover, for the functions
\begin{equation}
c_{N}(\omega,x)=\cos\omega x+2\sum_{n=0}^{[N/2]} (-1)^{n}\beta_{2n}(x)j_{2n}%
(\omega x) \label{cN}%
\end{equation}
and
\begin{equation}
s_{N}(\omega,x)=\sin\omega x+2\sum_{n=0}^{[(N-1)/2]} (-1)^{n}\beta_{2n+1}%
(x)j_{2n+1}(\omega x) \label{sN}%
\end{equation}
the following estimates hold%
\begin{equation}
\label{estc1}\left\vert c(\omega,x)-c_{N}(\omega,x)\right\vert \leq
2|x|\varepsilon_{N}(x)\qquad\text{and}\qquad\left\vert s(\omega,x)-s_{N}%
(\omega,x)\right\vert \leq2|x|\varepsilon_{N}(x)
\end{equation}
for any $\omega\in\mathbb{R}$, $\omega\ne0$, and
\begin{equation}
\label{estc2}\left\vert c(\omega,x)-c_{N}(\omega,x)\right\vert \leq
\frac{2\varepsilon_{N}(x)\,\sinh(Cx)}{C}\qquad\text{and}\qquad\left\vert
s(\omega,x)-s_{N}(\omega,x)\right\vert \leq\frac{2\varepsilon_{N}(x)\,\sinh
(Cx)}{C}%
\end{equation}
for any $\omega\in\mathbb{C}$, $\omega\neq0$ belonging to the strip
$\left\vert \operatorname{Im}\omega\right\vert \leq C$, $C\geq0$, where
$\varepsilon_N$ is a sufficiently small nonnegative function such that
$\left\vert K(x,t)-K_{N}(x,t)\right\vert \leq\varepsilon_N(x)$ which exists due
to Theorem \ref{Th transmutation kernel via beta and P} (an estimate for
$\varepsilon_N(x)$ is presented in \eqref{EstimateK ParSum}).
\end{theorem}

\begin{proof}
Substitution of $K(x,t)$ in the form of the series (\ref{K via beta}) into
(\ref{c and s definition}) leads to the equalities%
\begin{equation*}
c(\omega,x)    =\cos\omega x+\sum_{j=0}^{\infty}
\frac{\beta_{j}(x)}{x}\int_{-x}^{x}P_{j}\left(  \frac{t}{x}\right)  \cos\omega
t\,dt =\cos\omega x+\sum_{j=0}^{\infty}
\beta_{j}(x)\int_{-1}^{1}P_{j}\left(  y\right)  \cos\left(  \omega xy\right)
\,dy
\end{equation*}
and
\[
s(\omega,x)=\sin\omega x+\sum_{j=0}^{\infty}
\beta_{j}(x)\int_{-1}^{1}P_{j}\left(  y\right)  \sin\left(  \omega xy\right)
\,dy.
\]
Using formula 2.17.7 from \cite[p. 433]{Prudnikov},%
\[
\int_{0}^{a}\left\{
\begin{array}
[c]{l}%
P_{2n+1}\left(  \frac{y}{a}\right)  \cdot\sin by\\
P_{2n}\left(  \frac{y}{a}\right)  \cdot\cos by
\end{array}
\right\}  dy=\left(  -1\right)  ^{n}\sqrt{\frac{\pi a}{2b}}J_{2n+\delta
+1/2}(ab),\text{\quad}\delta=\left\{
\begin{array}
[c]{l}%
1\\
0
\end{array}
\right\}  ,\text{\quad}a>0,
\]
we obtain the representations (\ref{c(omega,x) via bessel}) and
(\ref{s(omega,x) via bessel}).

The convergence of the series with respect to $\omega$ can be established
using the fact that for each $x$ the series represent the Neumann series (see,
e.g., \cite{Watson} and \cite{Wilkins}). Indeed, the function $\omega\left(
c(\omega,x)-\cos\omega x\right)  $ regarded as a function of a complex
variable $\omega$ is entire and as the radius of convergence of the Neumann
series coincides \cite[pp. 524-526]{Watson} with the radius of convergence of
its associated power series (obtained from the SPPS representation) we obtain
that the series (\ref{c(omega,x) via bessel}) and (\ref{s(omega,x) via bessel}%
) converge uniformly on any compact subset of the complex plane of the
variable $\omega$.

Consider a complex $\omega\neq0$ belonging to the strip $\left\vert \operatorname{Im}%
\omega\right\vert \leq C$. We obtain
\begin{equation}
\label{Error cN2}%
\begin{split}
\vert c(\omega,x)  &  -c_{N}(\omega,x)\vert\leq\int_{-x}^{x}\left\vert
K(x,t)-K_{N}(x,t)\right\vert \left\vert \cos\omega t\right\vert \,dt\leq
2\varepsilon_{N}(x)\int_{0}^{x}\left\vert \cos\omega t\right\vert \,dt\\
&  \leq\varepsilon_{N}(x)\int_{0}^{x}\left(  e^{\operatorname{Im}\omega
t}+e^{-\operatorname{Im}\omega t}\right)  \,dt=2\varepsilon_{N}(x)\int_{0}%
^{x}\cosh\left(  \left\vert \operatorname{Im}\omega\right\vert \,t\right)
\,dt=\frac{2\varepsilon_{N}(x)\,\sinh(\left\vert \operatorname{Im}\omega
\right\vert x)}{\left\vert \operatorname{Im}\omega\right\vert }.
\end{split}
\end{equation}
Since the function $\sinh(\xi x)/\xi$ is monotonically increasing with respect
to both variables when $\xi,x\geq0$, we obtain the required inequality
(\ref{estc2}). The second inequality in \eqref{estc2} and the inequalities \eqref{estc1} are
proved similarly.

The uniform convergence of the series \eqref{c(omega,x) via bessel} and
\eqref{s(omega,x) via bessel} with respect to the variable $x$ follows
directly from the inequalities \eqref{estc1} and \eqref{estc2} and estimate
\eqref{EstimateK ParSum} valid at least for $p=0$.
\end{proof}

\begin{remark}
The inequalities \eqref{estc1} and \eqref{estc2} are of particular
importance when using representations \eqref{c(omega,x) via bessel} and
\eqref{s(omega,x) via bessel} for solving spectral problems for
\eqref{SL_omega} because they guarantee a uniform ($\omega$-independent)
approximation of eigendata (see \cite[Proposition 7.1]{KT AnalyticApprox})
which is illustrated by numerical experiments in Section \ref{Sect Num Exp}.
\end{remark}

\begin{remark}
In \cite{KT MMET 2012} another representation for solutions of
\eqref{SL_omega} was obtained in the form of Neumann series of Bessel
functions. It was based on the representation of the functions $\sin\omega x$
and $\cos\omega x$ as series in terms of Tchebyshev polynomials. To the
difference of the result of Theorem
\ref{Th Representation of solutions via Bessel} that idea did not lead to an
approximation of the solutions uniform with respect to $\omega$.
\end{remark}

Note that the numbers $j_k(z)$ for fixed $z$ rapidly decrease as $k\to\infty$, see, e.g., \cite[(9.1.62)]{Abramowitz}. Hence, the convergence rate of the series \eqref{c(omega,x) via bessel} and \eqref{s(omega,x) via bessel} for any fixed $\omega$ (and for bounded subsets $\Omega\subset\mathbb{C}$) is, in fact, exponential.
\begin{proposition}\label{Prop Exp Conv}
Let $x>0$ be fixed and $\omega\in\mathbb{C}$ satisfy $|\omega|\le \omega_0$. Suppose that $q\in C^{(p)}[0,b]$ for some $p\in \mathbb{N}_0$. Then for all $N>\max\{\omega_0 x,p\}/2$ the remainders of the series \eqref{c(omega,x) via bessel} and \eqref{s(omega,x) via bessel} satisfy
\begin{equation*}
    |c(\omega,x)-c_{2N}(\omega,x)|=|c(\omega,x)-c_{2N+1}(\omega,x)|\le \frac{c x^{p+2}e^{|\operatorname{Im}\omega|x}}{(2N+2)^{p+1/2}}\cdot\frac{1}{(2N+2)!}\cdot\left|\frac{\omega_0 x}{2}\right|^{2N+2},
\end{equation*}
and
\begin{equation*}
    |s(\omega,x)-s_{2N-1}(\omega,x)|=|s(\omega,x)-s_{2N}(\omega,x)|\le \frac{c x^{p+2}e^{|\operatorname{Im}\omega|x}}{(2N+1)^{p+1/2}}\cdot\frac{1}{(2N+1)!}\cdot\left|\frac{\omega_0 x}{2}\right|^{2N+1},
\end{equation*}
where $c$ is a constant depending on $q$ and $p$ only.
\end{proposition}

\begin{proof}
It follows from \eqref{int = beta} that
\begin{equation}\label{beta FL}
\beta_{n}(x)=\frac{2n+1}{2}\int_{-x}^{x}K(x,t)P_{n}\left(  \frac{t}{x}\right)
\,dt=\frac{2n+1}{2}\int_{-1}^{1}xK(x,xz)P_{n}(z)\,dz,
\end{equation}
i.e., for every $x>0$ the numbers $\beta_n(x)$ are the Fourier-Legendre coefficients of the function $g(z):=xK(x,xz)$. Since $q\in C^{(p)}[0,b]$, the function $g\in C^{(p+1)}[-1,1]$ and hence (see \cite[Corollary I to Theorem XIV]{Jackson} and \cite{WangXiang})
\begin{equation}\label{beta_n decay}
|\beta_n(x)|\le \frac{c_p V x^{p+2}}{n^{p+1/2}}, \qquad n>p,
\end{equation}
where $c_p>0$ is a universal constant and $V=\max_{x\le b, |t|\le x}|\partial_t^{p+1} K(x,t)|$. Combining the inequality \eqref{beta_n decay} with the inequality \cite[(9.1.62)]{Abramowitz}
\[
|j_n(z)|\le \sqrt{\pi}\left|\frac z2\right|^n \frac{e^{\operatorname{Im}z}}{\Gamma(n+3/2)}
\]
one easily obtains the announced estimates.
\end{proof}

\section{Representation for derivatives of
solutions\label{Sect Representation Derivatives}}

Differentiation of the equalities (\ref{c and s definition}) with respect to
$x$ and the Goursat conditions
\begin{equation}
\label{GoursatCond for K}K(x,x)=\frac h2 + \frac12\int_{0}^{x} q(s)\,ds,\qquad
K(x,-x)=\frac h2,\qquad0\le x\le b
\end{equation}
give us the relations
\begin{equation}
c^{\prime}(\omega,x) =-\omega\sin\omega x+\int_{-x}^{x}K_{1} (x,t)\cos\omega
t\,dt +\left(  h+\frac{1}{2}\int_{0}^{x}q(s)\,ds\right)  \cos\omega x
\label{c prime}%
\end{equation}
and
\begin{equation}
s^{\prime}(\omega,x) =\omega\cos\omega x+\int_{-x}^{x}K_{1}(x,t)\sin\omega
t\,dt +\frac{1}{2}\left(  \int_{0}^{x}q(s)\,ds\right)  \sin\omega x.
\label{s prime}%
\end{equation}
Here $K_{1}(x,t)$ is the derivative of $K(x,t)$ with respect to the first
variable.
To obtain a convenient representation for the kernel $K_{1}(x,t)$ we can apply
a procedure similar to that from Section
\ref{Sect Fourier-Legendre for the Kernel}. Let us seek $K_{1}(x,t)$ in the
form%
\begin{equation}
K_{1}(x,t)=%
{\displaystyle\sum\limits_{j=0}^{\infty}}
\frac{\gamma_{j}(x)}{x}P_{j}\left(  \frac{t}{x}\right)  .
\label{K1 via gammas}%
\end{equation}
Then analogously to (\ref{int = beta}) we have
\begin{equation}
\gamma_{n}(x)=\frac{2n+1}{2}\int_{-x}^{x}K_{1}(x,t)P_{n}\left(  \frac{t}%
{x}\right)  dt. \label{gamma n via K1}%
\end{equation}

Differentiation of (\ref{mapping property}) (and the use of
\eqref{GoursatCond for K}) gives us the relations%
\[
\int_{-x}^{x}K_{1}(x,t)t^{k}dt=\varphi_{k}^{\prime}(x)-kx^{k-1}-\frac{1}%
{2}\left(  \left(  1+(-1)^{k}\right)  h+\int_{0}^{x}q(s)\,ds\right)  x^{k}.
\]
Using (\ref{gamma n via K1}) we obtain then%
\begin{equation}
\label{gamma n}%
\begin{split}
\gamma_{n}(x)  &  =\frac{2n+1}{2}\sum_{k=0}^{n}\frac{l_{k,n}}{x^{k}}\int
_{-x}^{x}K_{1}(x,t)t^{k}dt\\
&  =\frac{2n+1}{2}\left(  \sum_{k=0}^{n}\frac{l_{k,n}\varphi_{k}^{\prime}%
(x)}{x^{k}}-\frac{n(n+1)}{2x}-\frac{1}{2}\int_{0}^{x}q(s)\,ds-\frac{h}%
{2}\left(  1+(-1)^{n}\right)  \right)
\end{split}
\end{equation}
where several elementary properties of Legendre polynomials (such as
$\sum_{k=0}^{n} l_{k,n}=P_{n}(1)=1$ and $\sum_{k=1}^{n} k l_{k,n}%
=P_{n}^{\prime}(1)=\frac{n(n+1)}2$) were employed.

Finally, if $f$ does not have zeros on $[0,b]$ (such $f$ always exists
\cite[Remark 5]{KrPorter2010} and \cite{Camporesi et al 2011}), the
derivatives $\varphi_{k}^{\prime}$ can be calculated by the formula%
\[
\varphi_{k}^{\prime}=k\psi_{k-1}+\frac{f^{\prime}}{f}\varphi_{k}%
\]
which follows directly from Definition \ref{Def Formal powers phik and psik}.
Otherwise it is convenient to use the formulas from \cite{KrTNewSPPS}.

In general, $K_{1}$ is a continuous function with respect to both variables
however we cannot guarantee additional smoothness of $K_{1}$ as function of
$t$ (moreover, similarly to Lemma
\ref{Lemma diff K and q} it is possible to show that belonging of
$K_{1}(x,\cdot)$ to some class $\operatorname{Lip}\alpha$ implies
$q\in\operatorname{Lip}\alpha$). And it is known that the Fourier-Legendre
series of a continuous function may not converge to the function even
pointwise. Nevertheless, the series always converges to the function in the
$L_{2}$ norm.

Denote by $K_{1,N}$ the partial sum of the series \eqref{K1 via gammas},
\[
K_{1,N}(x,t)=\sum_{j=0}^{N} \frac{\gamma_{j}(x)}{x}P_{j}\left(  \frac{t}%
{x}\right)  .
\]
Below we prove some estimates for the remainder $K_{1}-K_{1,N}$ with explicit
dependence on $x$.

Let $q\in C^{(p)}[0,b]$. Then $K_{1}$ is $p$ times continuously differentiable
with respect to $t$ which justifies the existence of the following constants.
Define
\[
M_{p}:=\max_{0\le x\le b,\ |t|\le x}\bigl| \partial^{p}_{t} K_{1}(x,t)\bigr|.
\]
and
\[
k_{0}(\delta):=\sup_{0\le \tau\le\delta} \sup_{\substack{0\le x\le b\\t_{1}%
,t_{2}\in[-x,x]:\, |t_{1}-t_{2}|\le \tau}}|K_{1}(x,t_{1})-K_{1}(x,t_{2})|,
\]
Since the kernel $K_{1}$ is a continuous function in the domain $0\le x\le b$,
$|t|\le x$, we have $k_{0}(\delta)\to0$, $\delta\to0$.

\begin{proposition}
\label{Pr Estimates K1} Suppose that $q\in C^{(p)}[0,b]$. Then for $p=0$,
$N\in\mathbb{N}$ and each $x\in(0,b]$
\begin{equation}
\Vert K_{1}(x,\cdot)-K_{1,N}(x,\cdot)\Vert_{L_{2}[-x,x]}\leq c_{0}\sqrt
{x}\cdot k_{0}\left(  \frac{x}{n}\right)  \leq c_{0}\sqrt{x}\cdot k_{0}\left(
\frac{b}{n}\right)  =o(1),\ N\rightarrow\infty, \label{EstK1 L2}%
\end{equation}
and for $p\geq1$, $0<x\leq b$, $|t|\leq x$
\begin{equation}
\bigl|K_{1}(x,t)-K_{1,N}(x,t)\bigr|\leq\frac{c_{p}M_{p}x^{p}}{N^{p-1/2}},\qquad N>p,
\label{EstK1 Uniform}%
\end{equation}
where the constants $c_{0}$ and $c_{p}$ do not depend on $q$ and $N$.
\end{proposition}

\begin{proof}
Theorem 6.2 from \cite{DeVoreLorentz} states that for a function $f\in
W_{2}^{r}[-1,1]$, $r\in\mathbb{N}_{0}$, the error $E_{n}(f)_{2}$ of the best
approximation of $f$ by polynomials of degree not exceeding $n$ (which in the
case of $L_{2}$ norm coincides with the partial sum of the Fourier-Legendre
series of $f$) satisfies
\begin{equation}
\label{Direct DeVore}E_{n}(f)_{2}\le\tilde c_{r} n^{-r} \omega(f^{(r)}%
,1/n)_{2},
\end{equation}
where the constant $\tilde c_{r}$ does not depend on $f$ and $\omega$ is the
modulus of continuity defined for a function $g\in L_{2}[a,b]$ as
\[
\omega(g,t)_{2}:=\sup_{0\le \tau\le t}\|g(x+\tau)-g(x)\|_{L_{2}[a,b-\tau]}.
\]

For a fixed $x>0$ consider a function $g(y):=K_{1}(x,xy)$, $-1\le y\le1$.
Then
\begin{equation}
\label{omega g}\omega\left(  g,\frac1n\right)  _{2} = \sup_{0\le \tau\le
1/n}\biggl(\int_{-1}^{1-\tau}\bigl|g(y+\tau)-g(y)\bigr|^{2} \,dy\biggr)^{1/2}\le
\sup_{0\le \tau\le1/n}\sqrt{2}k_{0}(x\tau)\le\sqrt{2}k_{0}\left(  \frac xn\right)  .
\end{equation}
Note that
\[
\|K_{1}(x,\cdot)-K_{1,N}(x,\cdot)\|_{L_{2}[-x,x]} = \sqrt{x}\|g-g_{N}%
\|_{L_{2}[-1,1]}=\sqrt{x} E_{n}(g)_{2},
\]
where $g_{N}$ is the partial sum of the Fourier-Legendre series for the
function $g$. Combining the last equality with \eqref{Direct DeVore} for $r=0$
and with \eqref{omega g} we obtain \eqref{EstK1 L2}.

The second inequality \eqref{EstK1 Uniform} can be obtained similarly to the
proof of Theorem \ref{Th Convergence rate Direct}.
\end{proof}

\begin{proposition}
The derivatives of the solutions $c(\omega,x)$ and $s(\omega,x)$ of equation
\eqref{SL_omega} admit the following representations
\begin{equation}
c^{\prime}(\omega,x)=-\omega\sin\omega x+\left(  h+\frac{1}{2}\int_{0}%
^{x}q(s)\,ds\right)  \cos\omega x+2\sum_{n=0}^{\infty}(-1)^{n}\gamma
_{2n}(x)j_{2n}(\omega x)\label{c prime (omega)}%
\end{equation}
and
\begin{equation}
s^{\prime}(\omega,x)=\omega\cos\omega x+\frac{1}{2}\left(  \int_{0}%
^{x}q(s)\,ds\right)  \sin\omega x+2\sum_{n=0}^{\infty}(-1)^{n}\gamma
_{2n+1}(x)j_{2n+1}(\omega x)\label{s prime (omega)}%
\end{equation}
where $\gamma_{k}$ are defined by \eqref{gamma n}. The series converge
uniformly for any $x$ from $[0,b]$ and converge uniformly with respect to
$\omega$ on any compact subset of the complex plane of the variable $\omega$.
\end{proposition}

\begin{proof}
The proof of these representations follows from (\ref{c prime}) and
(\ref{s prime}) by substitution of (\ref{K1 via gammas}) and similar procedure
as that from the proof of Theorem
\ref{Th Representation of solutions via Bessel}. The uniform convergence of
the series with respect to $x$ follows from Proposition \ref{Pr Estimates K1}
and Cauchy-Schwarz inequality. The uniform convergence with respect to
$\omega$ is proved analogously to Theorem
\ref{Th Representation of solutions via Bessel}.
\end{proof}

\begin{remark}
\label{Rem Estimates for derivatives}Consider the approximations of the
derivatives of the solutions
\begin{equation}
\overset{\circ}{c}_{N}(\omega,x)=-\omega\sin\omega x+\left(  h+\frac{1}{2}%
\int_{0}^{x}q(s)\,ds\right)  \cos\omega x+2\sum_{n=0}^{[N/2]} (-1)^{n}\gamma
_{2n}(x)j_{2n}(\omega x) \label{cN prime}%
\end{equation}
and
\begin{equation}
\overset{\circ}{s}_{N}(\omega,x)=\omega\cos\omega x+\frac{1}{2}\left(
\int_{0}^{x}q(s)\,ds\right)  \sin\omega x+2\sum_{n=0}^{[(N-1)/2]} (-1)^{n}%
\gamma_{2n+1}(x)j_{2n+1}(\omega x). \label{sN prime}%
\end{equation}
Let $\varepsilon_N$ be a sufficiently small nonnegative function such that
$\left\vert K_{1}(x,t)-K_{1,N}(x,t)\right\vert \leq\varepsilon_N(x)$. Then for
the differences between $c^{\prime}(\omega,x)$ and $\overset{\circ}{c}%
_{N}(\omega,x)$, as well as between $s^{\prime}(\omega,x)$ and $\overset
{\circ}{s}_{N}(\omega,x)$, we obtain exactly the same estimates as
\eqref{estc1} and \eqref{estc2}. Their proof is analogous.
\end{remark}

\begin{remark}
The representations proposed for solutions and their derivatives are not
limited to continuous potentials $q$. Consideration of the convergence of the
series \eqref{K via beta} and \eqref{K1 via gammas} in the $L_{2}$ norm is
possible for wider classes of potentials, up to $q\in W_{2}^{-1}(0,b)$, see
\cite{HrynivMykytyuk2003, HrynivMykytyuk2004} for the definitions and
construction of transmutation operators in that case.
\end{remark}

\section{A sequence of equations for $\beta_{k}$ and $\gamma_{k}%
$\label{Sect Sequence Equations}}

In this section we develop another recursive procedure for calculating the
functions $\beta_{k}$ and $\gamma_{k}$. It is based on the substitution of the
solutions $c(\omega,x)$ and $s(\omega,x)$ in the form
(\ref{c(omega,x) via bessel}) and (\ref{s(omega,x) via bessel}) into equation
(\ref{SL_omega}).

One of the advantages of this procedure is the improved stability for
numerical calculations. For example, the first formula from Definition
\ref{Def beta} includes the term
\[
\frac{p_{0k}}{p_{kk}}\beta_{0}=\frac{(2k+1)!}{k!(k+1)!2^{k}}\beta_{0}\sim
\frac{2^{k+1}}{\sqrt{\pi k}}\beta_{0},
\]
where we applied formulas from \cite[Chap. 4]{Suetin} and Stirling's formula.
Hence even the smallest error in the computation of $\beta_{0}$ leads to
exponentially growing errors in subsequent $\beta_{k}$'s. Errors in other
previous $\beta_{m}$'s are get multiplied as well. On the other hand, for any $x$ the
functions $\beta_{k}(x)/x$ are the Fourier-Legendre coefficients of a smooth
function $K(x,xt)$ and hence tend to
zero as $k\rightarrow\infty$ (see also Remark \ref{Rem Precision Control}). This explains why using Definition
\ref{Def beta} only a limited number of the functions $\beta_{k}$ can be
evaluated numerically. The procedure developed in this section presents
different recurrent formulas which allowed more functions $\beta_{k}$ to be
evaluated numerically in all experiments performed, see Example
\ref{Example Exp} for details.

Let us start with the solution $c(\omega,x)$ (see Theorem
\ref{Th Representation of solutions via Bessel}). Analogous formulas for
$s(\omega,x)$ are given below. We proceed formally and at the end of this
section justify for the case $q\in C^{2}[0,b]$ the possibility to
differentiate termwise all the series and explain why the final formulas
remain valid for the general case. Differentiating the solution $c(\omega,x)$ twice, using the formulas $j_{k}^{\prime}(z)=-j_{k+1}(z)+\frac{k}{z}j_{k}(z)$,
$k=0,1,\ldots$ (for the first derivative) and $j_{k}^{\prime}(z)=j_{k-1}(z)-\frac{k+1}{z}j_{k}(z)$,
$k=1,2,\ldots$ (for the second derivative) and
substituting into (\ref{SL_omega}) leads us to the equality
\begin{multline*}
\label{vspom 1}2\sum_{n=0}^{\infty}(-1)^{n}\left[ j_{2n}(\omega x)\left(
\beta_{2n}^{\prime\prime}(x)+\frac{4n}{x}\beta_{2n}^{\prime}(x)+\frac
{2n(2n-1)}{x^{2}}\beta_{2n}(x)\right)  +j_{2n+1}(\omega x)\left(
-2\omega\beta_{2n}^{\prime}(x)+\frac{2\omega}{x}\beta_{2n}(x)\right)  \right]
\\
=q(x)\left(  \cos\omega x+2\sum_{n=0}^{\infty}(-1)^{n}\beta_{2n}%
(x)j_{2n}(\omega x)\right)  .
\end{multline*}
Combining the terms containing $j_{0}(\omega x)$ and using $q(x)=2(\beta
_{0}^{\prime\prime}(x)-q(x)\beta_{0}(x))$ we obtain
\begin{equation}
\label{vspom 1a}%
\begin{split}
\bigl( \cos\omega x & -j_{0}(\omega x)\bigr)\bigl(\beta_{0}^{\prime\prime
}(x)-q(x)\beta_{0}(x)\bigr) = -2\omega\sum_{n=0}^{\infty}(-1)^{n}j_{2n+1}(\omega x)\left(  \beta
_{2n}^{\prime}(x)-\frac{1}{x}\beta_{2n}(x)\right)
\\
&  +\sum_{n=1}^{\infty}(-1)^{n}j_{2n}(\omega x)\left(  \beta_{2n}%
^{\prime\prime}(x)+\frac{4n}{x}\beta_{2n}^{\prime}(x)+\frac{2n(2n-1)}{x^{2}%
}\beta_{2n}(x)-q(x)\beta_{2n}(x)\right).\\
\end{split}
\end{equation}
The second series can be expressed in the terms of odd index spherical Bessel functions using the equality
\begin{equation}
\label{recusrive Bessel}j_{2n}(\omega x)=\frac{\omega x}{4n+1}\bigl( j_{2n-1}%
(\omega x)+j_{2n+1}(\omega x)\bigr).
\end{equation}
Note additionally that
\begin{equation}
\label{bessel1}\cos\omega x-j_{0}(\omega x) = \omega x\cdot j_{1}(\omega x).
\end{equation}
Applying \eqref{bessel1} and
\eqref{recusrive Bessel} to \eqref{vspom 1a} and dividing by $\omega x$ one
can see that \eqref{vspom 1a} can be written as
\begin{equation}
\sum_{n=1}^{\infty}\alpha_{n}(x)j_{2n-1}(\omega x) = 0,\label{vspom 1b}%
\end{equation}
where
\begin{equation*}
\label{alpha n}%
\begin{split}
\alpha_{n}(x)  &  =(-1)^{n}\left[  \frac{1}{4n+1}\left(  \beta_{2n}%
^{\prime\prime}(x)+\frac{4n}{x}\beta_{2n}^{\prime}(x)+\left(  \frac
{2n(2n-1)}{x^{2}}-q(x)\right)  \beta_{2n}(x)\right)  \right. \\
&  -\frac{1}{4n-3}\left(  \beta_{2(n-1)}^{\prime\prime}(x)+\frac{4\left(
n-1\right)  }{x}\beta_{2(n-1)}^{\prime}(x)+\left(  \frac{2(n-1)\left(
2(n-1)-1\right)  }{x^{2}}-q(x)\right)  \beta_{2(n-1)} (x)\right) \\
&  +2\left. \left(  \frac{1}{x}\beta_{2(n-1)}^{\prime}(x)-\frac{1}{x^{2}}%
\beta_{2(n-1)}(x)\right)  \right]  .
\end{split}
\end{equation*}

Multiplying the equality \eqref{vspom 1b} by $j_{2m-1}(\omega x)$,
$m=1,2,\ldots$, integrating with respect to $\omega$ from 0 to $\infty$ and
using the integral
\begin{equation}
\label{orthogonal Bessels}\int_{0}^{\infty}j_{\nu+2n}(y)j_{\nu+2m}(y)\,dy=0
\end{equation}
for $n,m\in\mathbb{Z}$ with $n\ne m$ and $m+n+\nu>-1/2$ (c.f.,
\cite[Formula 11.4.6]{Abramowitz}) we obtain that all coefficients $\alpha
_{n}$ are identically equal to zero.

In order to simplify the equations $\alpha_{n}(x)=0$, $n=1,2,\ldots$, consider
the functions
\[
\sigma_{2n}(x):=x^{2n}\beta_{2n}(x),\quad n=0,1,\ldots.
\]
Then equations $\alpha_{n}(x)=0$ take the form
\begin{equation}%
\begin{split}
\sigma_{2n}^{\prime\prime}(x)-q(x)\sigma_{2n}(x) &  =\frac{4n+1}{4n-3}%
x^{2}\left(  \sigma_{2\left(  n-1\right)  }^{\prime\prime}(x)-q(x)\sigma
_{2\left(  n-1\right)  }(x)\right)  \\
&  \quad-2\left(  4n+1\right)  x\left(  \sigma_{2(n-1)}^{\prime}%
(x)-\frac{2n-1}{x}\sigma_{2(n-1)}(x)\right)  .
\end{split}
\label{sequence of equations for sigmas 2n}%
\end{equation}

Equations similar to (\ref{sequence of equations for sigmas 2n}) can be
derived also for the odd coefficients. Calculation similar to that for
$c(\omega,x)$ leads to the equality
\begin{align*}
\frac{q(x)}{2}\frac{\sin\omega x}{\omega x} &  =\sum_{n=0}^{\infty}%
(-1)^{n}\left[  \frac{j_{2n}(\omega x)}{4n+3}\left(  \beta_{2n+1}%
^{\prime\prime}(x)+\frac{2(2n+1)}{x}\beta_{2n+1}^{\prime}(x)+\left(
\frac{2n(2n+1)}{x^{2}}-q(x)\right)  \beta_{2n+1}(x)\right)  \right.  \\
&  \quad+\frac{j_{2n+2}(\omega x)}{4n+3}\left(  \beta_{2n+1}^{\prime\prime
}(x)+\frac{2(2n+1)}{x}\beta_{2n+1}^{\prime}(x)+\left(  \frac{2n(2n+1)}{x^{2}%
}-q(x)\right)  \beta_{2n+1}(x)\right)  \\
&  \quad-\left.  2j_{2n+2}(\omega x)\left(  \frac{1}{x}\beta_{2n+1}^{\prime
}(x)-\frac{1}{x^{2}}\beta_{2n+1}(x)\right)  \right]  .
\end{align*}
Noting that $\sin(\omega x)/\omega x=j_{0}(\omega x)$ one can see that the
last equality is of the form
\[
\sum_{n=0}^{\infty}\alpha_{n}(x)j_{2n}(\omega x)=0,
\]
where
\begin{align*}
\alpha_{n}(x) &  =(-1)^{n}\left[  \frac{1}{4n+3}\left(  \beta_{2n+1}%
^{\prime\prime}(x)+\frac{2(2n+1)}{x}\beta_{2n+1}^{\prime}(x)+\left(
\frac{2n(2n+1)}{x^{2}}-q(x)\right)  \beta_{2n+1}(x)\right)  \right.  \\
&  \quad-\frac{1}{4n-1}\left(  \beta_{2n-1}^{\prime\prime}(x)+\frac
{2(2n-1)}{x}\beta_{2n-1}^{\prime}(x)+\left(  \frac{\left(  2n-2\right)
(2n-1)}{x^{2}}-q(x)\right)  \beta_{2n-1}(x)\right)  \\
&  \quad+2\left.  \left(  \frac{1}{x}\beta_{2n-1}^{\prime}(x)-\frac{1}{x^{2}%
}\beta_{2n-1}(x)\right)  \right]
\end{align*}
and we have taken $\beta_{-1}:=1/2$ to simplify notations for $n=0$. Applying
the integral \eqref{orthogonal Bessels} one obtains the relations $\alpha
_{n}\equiv0$ for $n=0,1,2,\ldots$. Introducing $\sigma_{2n+1}=x^{2n+1}%
\beta_{2n+1}(x)$ we rewrite them in the form
\[
\frac{1}{4n+3}\left(  \sigma_{2n+1}^{\prime\prime}(x)-q(x)\sigma
_{2n+1}(x)\right)  =\frac{x^{2}}{4n-1}\left(  \sigma_{2n-1}^{\prime\prime
}(x)-q(x)\sigma_{2n-1}(x)\right)  -2x\left(  \sigma_{2n-1}^{\prime}%
(x)-\frac{2n}{x}\sigma_{2n-1}(x)\right)  .
\]
Combining the even with the odd cases we obtain the following sequence of
equations to find coefficients $\beta_{n}(x)=x^{-n}\sigma_{n}(x)$ for the
representations of solutions
\begin{equation}
\frac{1}{2n+1}\left(  \sigma_{n}^{\prime\prime}(x)-q(x)\sigma_{n}(x)\right)
=\frac{x^{2}}{2n-3}\left(  \sigma_{n-2}^{\prime\prime}(x)-q(x)\sigma
_{n-2}(x)\right)  -2x\left(  \sigma_{n-2}^{\prime}(x)-\frac{n-1}{x}%
\sigma_{n-2}(x)\right)  .\label{sequence of equations for sigma}%
\end{equation}

To obtain the equations for the coefficients $\gamma_{k}$ one has to compare
\eqref{c prime (omega)} and \eqref{s prime (omega)}
with the derivatives of (\ref{c(omega,x) via bessel}) and (\ref{s(omega,x) via bessel}) and proceed similarly to the previous cases. As a result, the following relations can be obtained
\begin{align}
\gamma_{0}(x) &  =\beta_{0}^{\prime}(x)-\frac{h}{2}-\frac{1}{4}\int_{0}%
^{x}q(s)\,ds,\nonumber\\
\gamma_1(x)&=\frac{1}{x}\beta_1(x)+\beta_1'(x) -\frac 34\int_0^x q(s)\,ds,\nonumber\\
\gamma_{n}(x) &  =\frac{n}{x}\beta_{n}(x)+\beta_{n}^{\prime}(x)+\frac
{2n+1}{2n-3}\left(  \gamma_{n-2}(x)-\beta_{n-2}^{\prime}(x)+\frac{n-1}{x}%
\beta_{n-2}(x)\right)  ,\qquad n=2,3,\ldots
\label{sequence of equations for gamma}%
\end{align}
Note that the last formula holds for $n=1$ as well if we define $\gamma_{-1}:=\frac 14\int_0^x q(s)\,ds$.
Introducing notations $\tau_{n}(x):=x^{n}\gamma_{n}(x)$ we can rewrite
equation \eqref{sequence of equations for gamma} in terms of the functions
$\sigma_{n}$.
\begin{equation}
\tau_{n}(x)=\sigma_{n}^{\prime}(x)+\frac{2n+1}{2n-3}x^{2}\left(  \tau
_{n-2}(x)-\sigma_{n-2}^{\prime}(x)\right)  +(2n+1)x\sigma_{n-2}%
(x).\label{sequence of equations for tau}%
\end{equation}

Hence the construction of the functions $\beta_{n}$ and $\gamma_{n}$ for
$n=1,2,\ldots$ reduces to solution of a recurrent sequence of inhomogeneous
Schrödinger equations (\ref{sequence of equations for sigma}) having the
form
\begin{equation}
\label{NonHom Schrodinger}\sigma_{n}^{\prime\prime}(x)-q(x)\sigma_{n}%
(x)=h_{n}(x)
\end{equation}
with the initial conditions $\sigma_{n}(0)=\sigma_{n}^{\prime}(0)=0$.

Thus, the following statement is proved.

\begin{proposition}
\label{Prop Beta and Gamma recurrent} The functions $\sigma_{n}(x):=x^{n}%
\beta_{n}(x)$ where $\beta_{n}$ are the coefficients from
\eqref{c(omega,x) via bessel} and \eqref{s(omega,x) via bessel} satisfy the
sequence of recurrent differential equations
\eqref{sequence of equations for sigma} for $n=1,2,\ldots$ with the initial
conditions $\sigma_{n}(0)=\sigma_{n}^{\prime}(0)=0$ and with the first
functions given by $\beta_{-1}:=1/2$ and $\beta_{0}=(f-1)/2$. The functions
$\tau_{n}(x):=x^{n} \gamma_{n}(x)$ where $\gamma_{n}$ are the coefficients
from \eqref{c prime (omega)} and \eqref{s prime (omega)} are given by the
sequence of recurrent relations \eqref{sequence of equations for tau} with the
first functions given by $\gamma_{-1}:=\frac 1{4}\int_0^x q(s)\,ds$ and $\gamma_{0}=\frac{f^{\prime}-h}2 -
\frac14\int_{0}^{x} q(s)\,ds$.
\end{proposition}

\begin{remark}
The values $\sigma_1(x)=\frac 32(\varphi_1(x)-x)$ and $\tau_1(x)=\frac 32\left(\frac{f'\varphi_1+1}{f} -1-\frac x2 \int_0^x q(s)\,ds\right)$ can also be used as the initial values for Proposition \ref{Prop Beta and Gamma recurrent}.
\end{remark}

\begin{remark}
Let $L:=\partial^2-q(x)$. Equations \eqref{sequence of equations for sigma} and \eqref{sequence of equations for gamma} can be written in the following somewhat more symmetric form
\begin{align*}
    \frac{1}{x^{n}}L\left[x^n \beta_n(x)\right] &= \frac{2n+1}{2n-3} x^{n-1} L\left[\frac{\beta_{n-2}(x)}{x^{n-1}}\right],\\
    \gamma_n(x)-\frac 1{x^n}\left(x^n\beta_n(x)\right)' &= \frac{2n+1}{2n-3}\left[\gamma_{n-2}(x)-x^{n-1}\left(\frac{\beta_{n-2}(x)}{x^{n-1}}\right)'\right].
\end{align*}
\end{remark}

The solution of the inhomogeneous Schrödinger equations
\eqref{NonHom Schrodinger} with the initial conditions $\sigma_{n}%
(0)=\sigma_{n}^{\prime}(0)=0$ can be taken in the form (c.f.,
\cite{KrPorter2010})
\[
\sigma_{n}(x)=f(x)\int_{0}^{x}\left(  \frac{1}{f^{2}(s)}\int_{0}^{s}%
f(t)h_{n}(t)dt\right)  ds.
\]
Substituting the right-hand side from equation
\eqref{sequence of equations for sigma} and performing several integrations by
parts to get rid of the derivatives of the function $\sigma_{n-2}$ under the
integral signs we obtain the following recurrent formulas for the functions
$\sigma_{n}$ and $\tau_{n}$.
\begin{align}
\eta_{n}(x) &  =\int_{0}^{x}\bigl(tf^{\prime}(t)+(n-1)f(t)\bigr)\sigma
_{n-2}(t)\,dt,\quad\quad\theta_{n}(x)=\int_{0}^{x}\frac{1}{f^{2}(t)}%
\bigl(\eta_{n}(t)-tf(t)\sigma_{n-2}(t)\bigr)dt,\nonumber\\
\sigma_{n}(x) &  =\frac{2n+1}{2n-3}\left[  x^{2}\sigma_{n-2}%
(x)+c_n f(x)\theta_{n}(x)\right]  ,\label{sigma n}\\
\tau_{n}(x) &  =\frac{2n+1}{2n-3}\left[  x^{2}\tau_{n-2}(x)+
c_n\left(f^{\prime}(x)\theta_{n}(x)+\frac{\eta_{n}(x)}{f(x)}\right)-(c_n-2n+1)x\sigma_{n-2}(x)
\right]  ,\quad n=1,2,\ldots\label{tau n}%
\end{align}
where $c_{n}=1$ if $n=1$ and $c_n=2(2n-1)$ otherwise.

Now we explain why the series \eqref{c(omega,x) via bessel} and
\eqref{s(omega,x) via bessel} can be differentiated termwise. Suppose that
$q\in C^{2}[0,b]$. First, it follows from the equality $\sum_{n=0}^{\infty
}(2n+1)j_{n}^{2}(z)=1$ (\cite[10.1.50]{Abramowitz}) and Cauchy-Schwarz
inequality that a series $\sum_{n=0}^{\infty}a_{n}(x)j_{2n+\delta}(\omega x)$
(where $\delta$ is zero or one) is uniformly convergent provided that the
series $\sum_{n=0}^{\infty}\frac{a_{n}^{2}(x)}{n}$ is uniformly convergent.
Second, it follows from \cite[Corollary I to Theorem XIV]{Jackson} and
\cite{WangXiang} that for a function $g\in C^{(p+1)}[-1,1]$ its
Fourier-Legendre coefficients $a_{n}(g)$ satisfy
\[
|a_{n}(g)|\leq\frac{c_{p}V}{n^{p+1/2}},
\]
where $c_{p}$ is a universal constant and $V=\max_{[-1,1]}|g^{(p+1)}(x)|$.

Consider coefficients $\beta_{n}$. As can be seen from
\eqref{beta direct definition}, $\beta_{n}\in C^{2}[0,b]$ (at $x=0$ we define
$\beta_{n}$ by continuity). Moreover, $\beta_{n}(x)$ are the Fourier-Legendre coefficients of the function
$xK(x,xz)\in C^{(3)}[-1,1]$ (with respect to $z$, see \eqref{beta FL}). Hence $|\beta_{n}(x)|\leq
c_{3}n^{-5/2}$. For the derivatives we have
\[
\beta_{n}^{\prime}(x)=\frac{2n+1}{2}\int_{-1}^{1}\bigl(K(x,xz)+xK_{1}%
(x,xz)+xzK_{2}(x,xz)\bigr)P_{n}(z)\,dz,
\]
i.e., $\beta_{n}^{\prime}(x)$ are the Fourier-Legendre coefficients of the
function $K(x,xz)+xK_{1}(x,xz)+xzK_{2}(x,xz)\in C^{(2)}[-1,1]$. Hence
$|\beta_{n}^{\prime}(x)|\leq c_{2}n^{-3/2}$. Similarly, $|\beta_{n}%
^{\prime\prime}(x)|\leq c_{1}n^{-1/2}$, providing the uniform convergence of
all series involved in this section.

The validity of the formulas \eqref{sigma n} and \eqref{tau n} in the general
case can be verified by taking a sequence $q_{n}\in C^{2}[0,b]$ such that
$q_{n}\to q$ uniformly as $n\to\infty$, constructing corresponding integral
kernels and coefficients $\beta_{n}$, $\gamma_{n}$ for each $q_{n}$ and
passing to the limit in the formulas \eqref{sigma n} and \eqref{tau n}. The
validity of Proposition \ref{Prop Beta and Gamma recurrent} now follows by
differentiating \eqref{sigma n} and \eqref{tau n}.

\section{Numerical solution of spectral problems}

\label{Sect Num Exp}

The representations for solutions and their derivatives
(\ref{c(omega,x) via bessel}), (\ref{s(omega,x) via bessel}) and
(\ref{c prime (omega)}), (\ref{s prime (omega)}) lend themselves for numerical
solving of equation (\ref{SL_omega}) and in particular for numerical solving
of related spectral problems. As an example, let us consider the
Sturm-Liouville problem for (\ref{SL_omega}),
\begin{align}
\alpha_{0}y(0)+\mu_{0}y^{\prime}(0)  &  =0,\label{SLBC0}\\
\alpha_{b}y(b)+\mu_{b}y^{\prime}(b)  &  =0, \label{SLBCb}%
\end{align}
where we allow the coefficients $\alpha_{0}$, $\mu_{0}$, $\alpha_{b}$ and
$\mu_{b}$ to be not only constants but also entire functions of the square
root $\omega$ of the spectral parameter $\lambda$ satisfying $|\alpha
_{0}|+|\mu_{0}|\neq0$ and $|\alpha_{b}|+|\mu_{b}|\neq0$ (for every $\lambda$).

Based on the results of the previous sections and taking into account that the solutions $c(\omega, x)$ and $s(\omega, x)$ satisfy the following initial conditions
\begin{align*}
    c(\omega, 0) &= 1, & s(\omega,0) &= 0,\\
    c'(\omega, 0) &= h, & s'(\omega,0) &= \omega,
\end{align*}
we can formulate the following
algorithm for solving spectral problems \eqref{SLBC0}--\eqref{SLBCb} for
equation (\ref{SL_omega}).

\begin{enumerate}
\item Find a non-vanishing on $[0,b]$ solution $f$ of the equation \eqref{SLhom}.
Let $f$ be normalized as $f(0)=1$ and define $h:=f^{\prime}(0)$. The solution
$f$ can be constructed using the SPPS representation, see, e.g.,
\cite{KrPorter2010} for details or using any other numerical method.


\item Compute the functions $\beta_{k}$ and $\gamma_{k}$, $k=0,\ldots,N$ using
\eqref{sigma n} and \eqref{tau n}.

\item Calculate the approximations $c_{N}(\omega,x)$ and $s_{N}(\omega,x)$ of
the solutions $c(\omega,x)$ and $s(\omega,x)$ by (\ref{cN}) and (\ref{sN}). If
necessary, calculate the approximations of the derivatives of the solutions
using (\ref{cN prime}) and (\ref{sN prime}).

\item The eigenvalues of the problem (\ref{SL_omega}), (\ref{SLBC0}),
\eqref{SLBCb} coincide with the squares of the zeros of the entire function
\begin{equation}
\Phi(\omega):=\alpha_{b}\left(\mu_{0}c(\omega,b)-(\alpha_{0}+\mu_{0}%
h)\frac{s(\omega,b)}\omega\right)+\mu_{b}\left(\mu_{0}c^{\prime}(\omega,b)-(\alpha_{0}%
+\mu_{0}h)\frac{s^{\prime}(\omega,b)}\omega\right) \label{SLCharEq}%
\end{equation}
and are approximated by squares of zeros of the function
\begin{equation}
\Phi_{N}(\omega):=\alpha_{b}\left(\mu_{0}c_{N}(\omega,b)-(\alpha_{0}+\mu
_{0}h)\frac{s_{N}(\omega,b)}\omega\right)+\mu_{b}\left(\mu_{0}\overset{\circ}{c}_{N}%
(\omega,b)-(\alpha_{0}+\mu_{0}h)\frac{\overset{\circ}{s}_{N}(\omega,b)}\omega\right).
\label{SLCharEqApprox}%
\end{equation}

\item The eigenfunction $y_{\lambda}$ corresponding to the eigenvalue
$\lambda=\omega^{2}$ can be taken in the form
\begin{equation}
y_{\lambda}=\mu_{0}c(\omega,x)-(\alpha_{0}+\mu_{0}h)\frac{s(\omega,x)}\omega.
\label{SLEigenfunction}%
\end{equation}
Hence once the eigenvalues are calculated the computation of the corresponding
eigenfunctions can be done using formulas \eqref{cN} and \eqref{sN}.
\end{enumerate}

We have applied the proposed algorithm both in machine precision (in Matlab 2012) and in arbitrary precision arithmetics (in Mathematica 8.0). We refer the reader to \cite[Section 7]{KT AnalyticApprox} for some
implementation details concerning the computation of the system $\varphi_k$ and related numerical integration aspects. Though, as compared to \cite{KT AnalyticApprox}, for the arbitrary precision arithmetics computation we used the modification of Clenshaw-Curtis integration method by Filippi \cite{Filippi} (see also \cite[Section 6.4]{Rabinowitz} and \cite{Trefethen}) to calculate the functions $\varphi_k$, $\psi_k$, $\sigma_k$ and $\tau_k$. This method is reportedly more accurate for computing indefinite integrals, and we illustrate its performance in Example \ref{Example Exp}. Note that this method can be efficiently realized via Fast Fourier transform and for the improved accuracy the last coefficient $a_N$ in \cite[(6.4.8) and (6.4.10)]{Rabinowitz} has to be halved, c.f., \cite[(2.13.1.10) and (2.13.1.11)]{Rabinowitz}.
For the machine precision calculations we used Newton-Cottes 6 point integration rule.

The independent evaluation of the spherical Bessel functions $j_k(\omega b)$ for all values of $k$ and all values of $\omega b$ using built-in routines from Matlab or Wolfram Mathematica can be rather slow. In order to speed up the evaluation of the series \eqref{cN}, \eqref{sN}, \eqref{cN prime} and \eqref{sN prime} the recurrent relations \eqref{recusrive Bessel} can be used. We refer the reader to \cite{Barnett}, \cite{GillmanFiebig} and references therein for further details.

Some numerical methods (such as Newton's method) can benefit from the knowledge of the $\omega$-derivative of the solutions and their derivatives. The representations for solutions and their derivatives
(\ref{c(omega,x) via bessel}), (\ref{s(omega,x) via bessel}) and
(\ref{c prime (omega)}), (\ref{s prime (omega)}) can be easily differentiated with respect to $\omega$. For example, for the solutions $c(\omega,x)$ and $s(\omega, x)$ one obtains
\begin{align*}
    c'_\omega(\omega, x) &= - x\sin \omega x + 2\sum_{n=0}^\infty (-1)^n \beta_{2n}(x) \left(\frac {2n}\omega j_{2n}(\omega x) - x j_{2n+1}(\omega x)\right),\\
    s'_\omega(\omega, x) &=  x\cos \omega x + 2\sum_{n=0}^\infty (-1)^n \beta_{2n+1}(x) \left( x j_{2n}(\omega x) - \frac {2n+2}\omega j_{2n+1}(\omega x)\right).
\end{align*}
The partial sums of the obtained series can be used as approximations to the $\omega$-derivatives. Convergence estimates of the partial sums can be obtained similarly to the previous sections.

Note that the coefficients $\{\beta_{j}\}$ and
$\{\gamma_{j}\}$ as Fourier coefficients of smooth functions decrease to zero
(not necessarily monotonically) as $j\to\infty$. The formulas presented in
Definition \ref{Def beta} as well as \eqref{beta direct definition} and
\eqref{gamma n via K1} involve the dependence on all preceding functions and hence present computational difficulties due to the limited computation precision and cancelation of comparable terms. In Examples
\ref{Example Exp}--\ref{Example GelfLev} we illustrate this and show that the alternative formulas introduced in Section \ref{Sect Sequence Equations} allow one to compute more coefficients $\beta_k$ and $\gamma_k$.
The following observation can be used to
estimate an optimal number $N$ to choose. 

\begin{remark}
\label{Rem Precision Control} The boundary conditions
\eqref{GoursatCond for K} offer a simple and efficient way for controlling the
accuracy of the numerical method. Indeed, substitution of \eqref{K via beta}
into \eqref{GoursatCond for K} leads to the equalities
\[
\sum_{j=0}^{\infty} \frac{\beta_{j}(x)}{x}=\frac{h}{2}+\frac{1}{2}\int_{0}%
^{x}q(s)\,ds\quad\text{and}\quad\sum_{j=0}^{\infty} \left(  -1\right)
^{j}\frac{\beta_{j}(x)}{x}=\frac{h}{2}
\]
(due to the relations $P_{j}\left(  1\right)  =1$ and $P_{j}\left(  -1\right)
=\left(  -1\right)  ^{j}$). The differences
\begin{equation}\label{ErrorCheck}
\varepsilon_{1,N}(x):=\biggl\vert\sum_{j=0}^{N} \frac{\beta_{j}(x)}{x}-\left(
\frac{h}{2}+\frac{1}{2}\int_{0}^{x} q(s)\,ds\right)  \biggr\vert \quad
\text{and}\quad\varepsilon_{2,N} (x):=\biggl\vert\sum_{j=0}^{N} \left(
-1\right)  ^{j}\frac{\beta_{j}(x)}{x}-\frac{h}{2}\biggr\vert
\end{equation}
indicate the accuracy of the approximation of the transmutation kernel and
hence the accuracy of the approximate solutions \eqref{cN} and \eqref{sN}.

Similarly, the accuracy of the coefficients $\gamma_k$ and the approximations \eqref{cN prime} and \eqref{sN prime} can be estimated using \eqref{K1 via gammas} and the following relations \cite{KT FuncProp},
\[
K_1(x,x)=\frac 14\biggl(q(x) + h\int_0^x q(s)\,ds+\frac 12\biggl(\int_0^x q(s)\,ds\biggr)^2\biggr), \qquad K_1(x,-x) = \frac14\biggl(q(0) + \int_0^x q(s)\,ds\biggr).
\]
\end{remark}

The results of the previous section allow us to prove the uniform error bound
for all approximate zeros of the characteristic function (at least when the
coefficients in the boundary conditions \eqref{SLBC0} and \eqref{SLBCb} are
independent of the spectral parameter) obtained by the proposed algorithm and
that neither spurious zeros appear nor zeros are missed. For the proof we refer to \cite[Section 7]{KT AnalyticApprox}.

The proposed algorithm is based on the exact analytical representation of the solutions \eqref{c(omega,x) via bessel}, \eqref{s(omega,x) via bessel} and their derivatives \eqref{c prime (omega)}, \eqref{s prime (omega)}. It can be easily combined with the widely used techniques such as the interval subdivision and the shooting method \cite{PryceBook}. However we decided to perform the numerical experiments globally without any interval subdivision, to illustrate that even applied directly the algorithm provides accurate eigendata.

\begin{example}
\label{Example Exp} Consider the following spectral problem (the first Paine
problem, \cite{Paine}, see also \cite[Example 7.4]{KT AnalyticApprox})
\[%
\begin{cases}
-u^{\prime\prime}+e^{x}u=\lambda u, \quad 0\leq x\leq\pi,\\
u(0,\lambda)=u(\pi,\lambda)=0.
\end{cases}
\]
\begin{figure}[htb!]
\centering
\includegraphics[bb=97 308 513 482, width=5.7778in,height=2.4167in]
{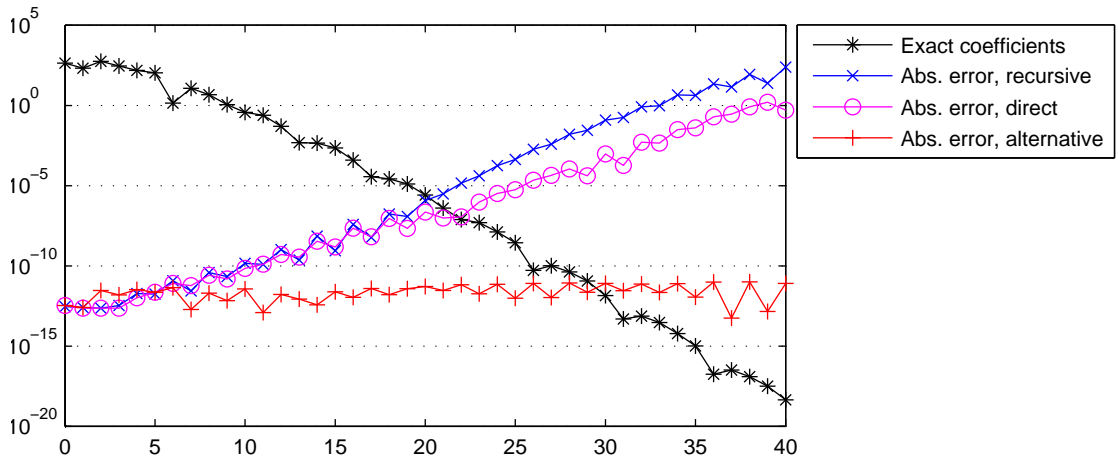}\caption{The plot of the absolute values of the coefficients
$\beta_{k}(\pi)$, $k\le40$ from Example \ref{Example Exp} (black line with
asterisks) together with absolute errors obtained using formulas from
Definition \ref{Def beta} (blue line with `x' marks), formulas
\eqref{beta direct definition} (magenta line with `o' marks) and formulas
\eqref{sigma n} (red line with `+' marks).}%
\label{Ex1Fig1}%
\end{figure}

\begin{figure}[htb!]
\centering
\includegraphics[bb=198 301 414 489, width=3in,height=2.6in]
{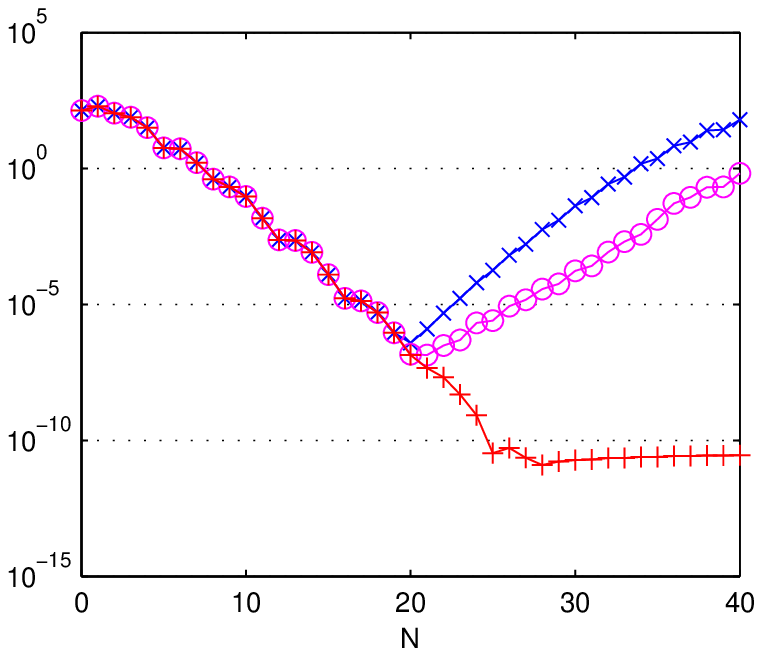}\ \ \  \includegraphics[bb=198 301 414 489, width=3in,height=2.6in]
{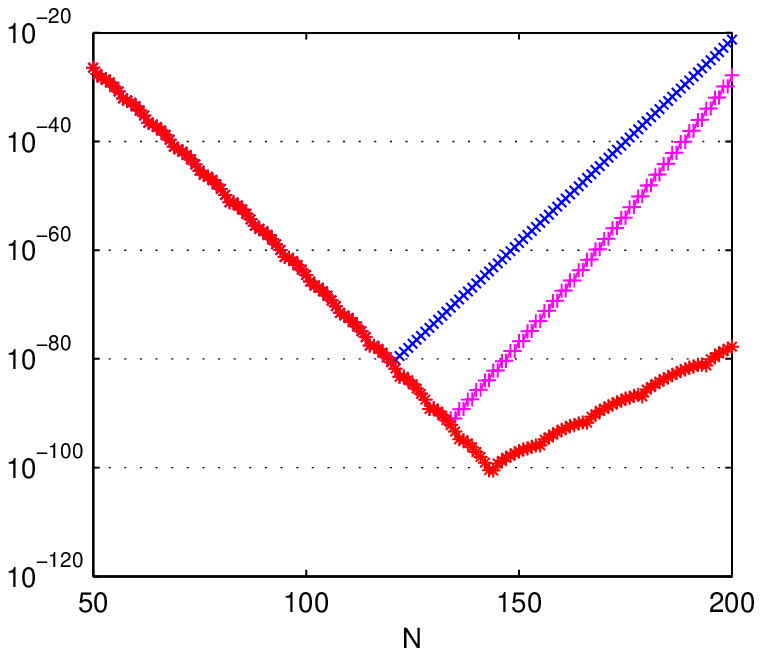}
\caption{The plot of the maximum of the differences \eqref{ErrorCheck} at the point $x=\pi$ from Example \ref{Example Exp}. The left graph corresponds to computation in the machine precision, the coefficients $\beta_k$, $k\le 40$ obtained using formulas from
Definition \ref{Def beta} (blue line with `x' marks), formulas
\eqref{beta direct definition} (magenta line with `o' marks) and formulas
\eqref{sigma n} (red line with `+' marks). The right graph corresponds to computation in high precision arithmetics, the coefficients $\beta_k$, $k\le 200$ obtained using formulas from
Definition \ref{Def beta} (blue line with `x' marks) and formulas
\eqref{sigma n} (magenta line with '+' marks computed with the use of Clenshaw-Curtis integration method, red line with `+' marks computed  with Filippi's modification).}
\label{Ex1Fig2}
\end{figure}

\begin{figure}[htb!]
\centering
\includegraphics[bb=126 316 486 475, width=5in,height=2.2in]
{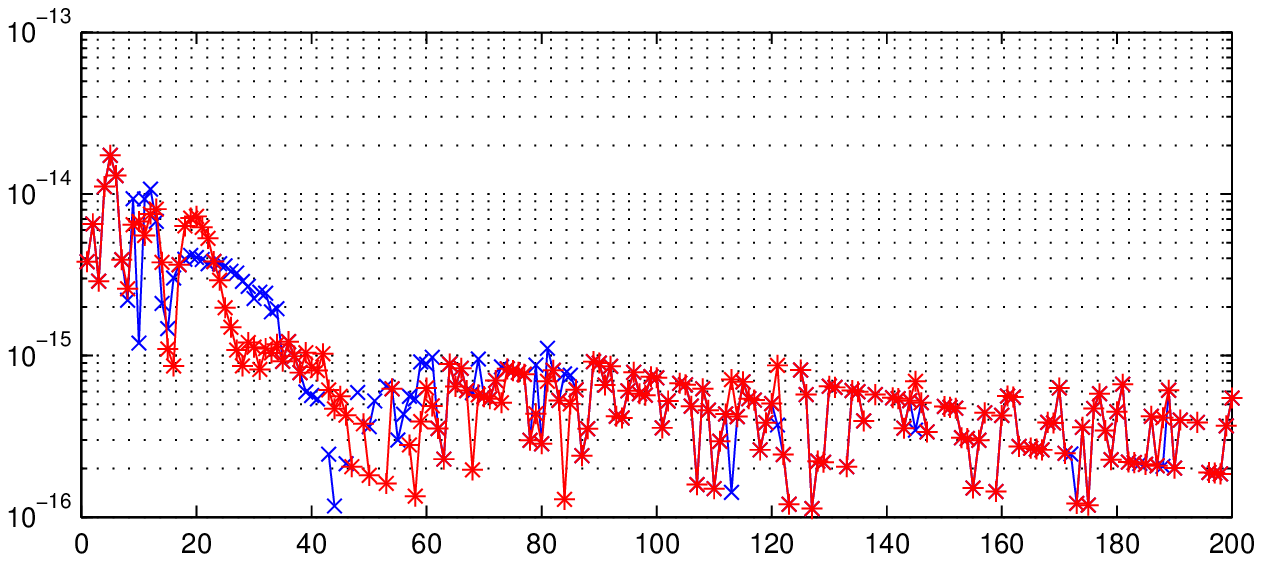}\\
\medskip
\includegraphics[bb=126 316 486 475, width=5in,height=2.2in]{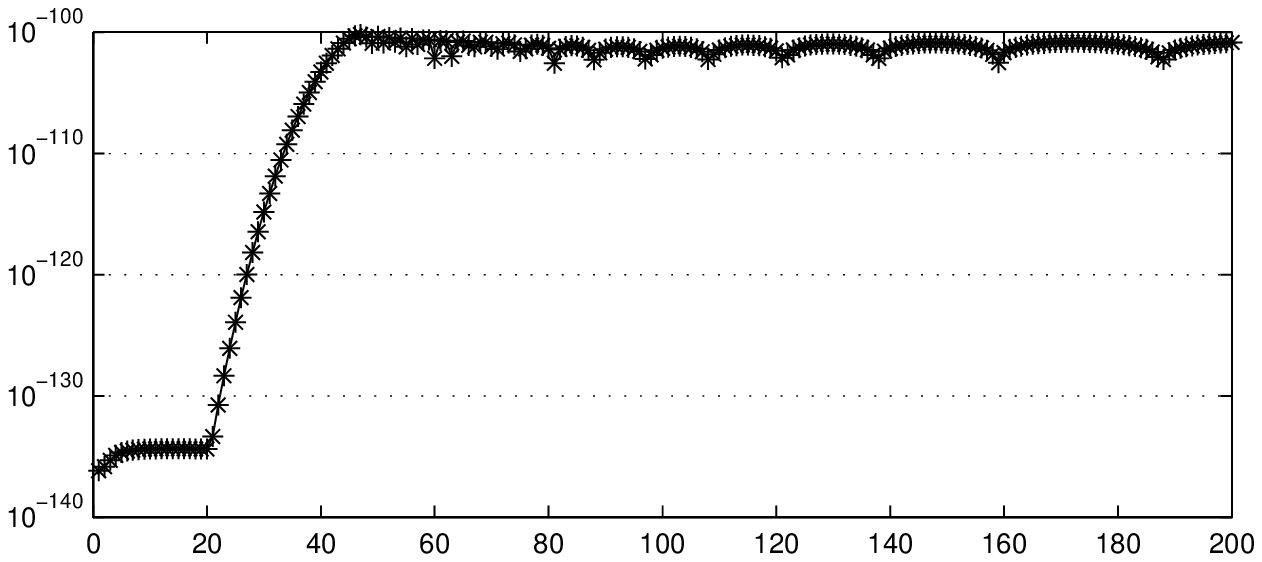}
\caption{Errors of the first 200 eigenvalues from Example \ref{Example Exp}. Top plot: relative errors,
obtained using machine precision and $N=29$ (blue line with `x' marks) and $N=40$ (red line with `*' marks). Bottom plot: absolute errors, obtained using high precision and $N=144$.}
\label{Ex1Fig3}
\end{figure}

For the Matlab program we computed functions $\beta_k$, $k\le 40$. All functions were represented by their values in 20001 uniformly spaced points. The modified Newton-Cottes 6 point integration rule was used to compute all integrals involved. In Mathematica we computed $\beta_k$ for $k\le 200$ using 200 digit arithmetics and representing all functions by their 257 values at  Tchebyshev-spaced points, the Filippi modification of Clenshaw-Curtis formula was used for the numerical integration. As a particular solution we took $f(x)=I_0(2e^{x/2})$, however we did not use the explicit formula computing instead this particular solution numerically from the SPPS representation. Once again we would like to emphasize the excellent performance of the SPPS representation, the calculated particular solution coincided with the one provided by the exact formula up to Mathematica's 200 digit accuracy.

On Figure \ref{Ex1Fig1} we present the absolute errors of the coefficients $\beta_k$ at $x=\pi$ computed in the machine precision using formulas from Definition \ref{Def beta}, \eqref{beta direct definition} and \eqref{sigma n}. As the exact values, the coefficients evaluated in Mathematica were used. As one can see, formula \eqref{sigma n} performed much better, coefficient errors remain of essentially the same order while two other formulas produce exponential error growth.

As we mentioned before, formula \eqref{ErrorCheck} from Remark \ref{Rem Precision Control} can be used to estimate the number $N$ of coefficients $\beta_k$ computed correctly. On Figure \ref{Ex1Fig2} we present the maximum of two differences from \eqref{ErrorCheck} evaluated at $x=\pi$ both in the machine precision and in the arbitrary precision arithmetics. The minimums on the first plot correlates with Figure \ref{Ex1Fig1}. Again, a better performance of the formula \eqref{sigma n} can be appreciated.
For the arbitrary precision arithmetics, the minimum at $N=144$ can be clearly seen and one can appreciate the better performance of the Filippi integration method in comparison with the Clenshaw-Curtis' one.  For the machine precision, the graph almost stabilizes at $N=29$.

On Figure \ref{Ex1Fig3} we present the errors of the computed eigenvalues. For the machine precision we have taken $N=29$ and $N=40$ to illustrate that the proposed method is not sensible to the value of $N$ while one chooses $N$ from the stabilized part of the differences \eqref{ErrorCheck}. As one can appreciate, the relative errors are close to the machine precision limit. For the arbitrary precision arithmetics we have taken $N=144$ (an optimal number determined from Figure \ref{Ex1Fig2}). As one can see, the eigenvalue errors remain uniformly bounded. The better precision of the first eigenvalues is explained by Proposition \ref{Prop Exp Conv}.


The computation time on a PC equipped with Intel i7-3770 microprocessor was:  in machine precision -- 0.25 seconds for constructing a particular solution and the coefficients $\beta_k$, $k\le 40$ and 0.63 seconds for finding 500 eigenvalues with the help of secant method; in arbitrary precision arithmetics a particular solution and the coefficients $\beta_k$, $k\le 200$ were computed in 9 seconds, 13 seconds more were necessary for our code based on Mathematica's function \texttt{FindRoot} and recurrent relation \eqref{recusrive Bessel} to find 500 eigenvalues. The maximum relative error of the first 500 eigenvalues computed in the machine precision was $2\cdot 10^{-14}$, the maximum absolute error of the first 500 eigenvalues computed in the high precision was $6.3\cdot 10^{-101}$.
\end{example}

\begin{example}\label{Example Paine}
Consider the
following spectral problem (the second Paine problem, \cite{Paine}, see also \cite[Example 7.5]{KT AnalyticApprox})
\begin{equation*}
\begin{cases}
-u''+\frac{1}{(x+0.1)^2} u=\lambda u, \quad 0\le x\le \pi,\\
u(0,\lambda)= u(\pi,\lambda)=0.
\end{cases}
\end{equation*}
\begin{figure}[htb!]
\centering
\includegraphics[bb=126 309 486 482, width=5in,height=2.4in]
{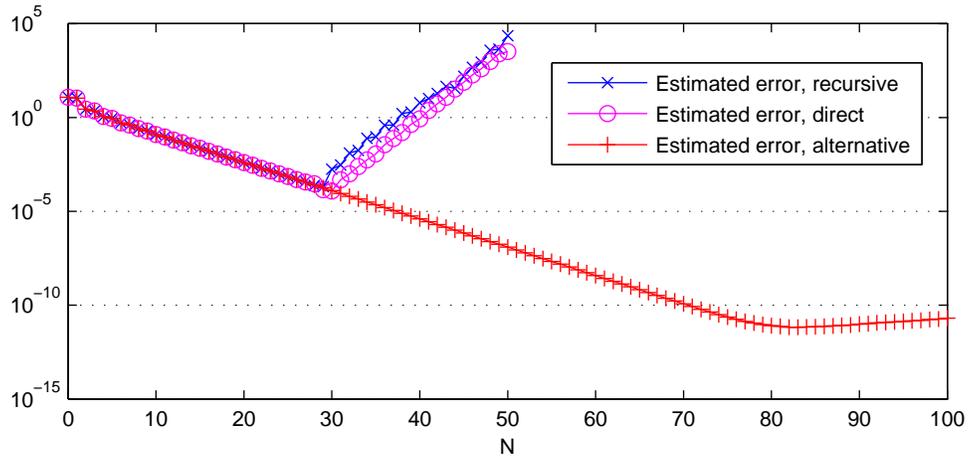}\caption{The plot of the maximum of the differences \eqref{ErrorCheck} at the point
$x=\pi$ from Example \ref{Example Paine}. The coefficients $\beta_k$, $k\le 100$ obtained using formulas from
Definition \ref{Def beta} (blue line with `x' marks), formulas
\eqref{beta direct definition} (magenta line with `o' marks) and formulas
\eqref{sigma n} (red line with `+' marks).}
\label{Ex2Fig1}
\end{figure}

\begin{figure}[htb!]
\centering
\includegraphics[bb=126 316 486 475, width=5in,height=2.2in]
{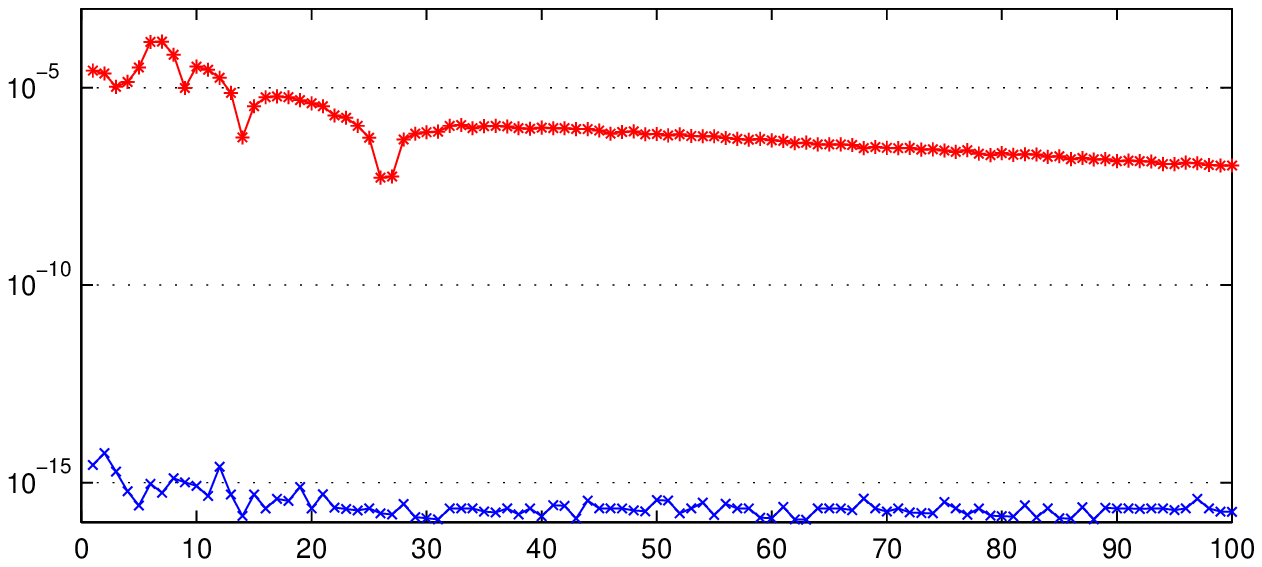}
\caption{Relative errors of the first 100 eigenvalues from Example \ref{Example Paine}. Upper red line with `*' marks: the best results we were able to obtain using machine precision and the method from \cite{KT AnalyticApprox}, lower blue line with `x' marks: results obtained by the proposed method, $N=83$.}
\label{Ex2Fig2}
\end{figure}

This problem was considered in \cite[Example 7.5]{KT AnalyticApprox}, where the results were reported for high precision arithmetics only. The reason was in the largeness of the coefficients arising in the solution of the approximation problem limiting the achievable accuracy of the approximation. As a result, the saturation occurred starting from $N=20$ and the eigenvalues were obtained with the error of about $0.001$ or worse, see Figure \ref{Ex2Fig2}.

The method proposed in this work allowed us to compute 84 coefficients $\beta_k$, see Figure \ref{Ex2Fig1}, and 500 eigenvalues were calculated. On Figure \ref{Ex2Fig2} we present relative errors of the first 100 eigenvalues. The maximum relative error was $5.6\cdot 10^{-15}$. The computation time was:  0.69 seconds for constructing a particular solution and the coefficients $\beta_k$, $k\le 100$, and 0.8 seconds for finding 500 eigenvalues with the help of the secant method.
\end{example}

\begin{example}\label{Example GelfLev}
Consider the
following spectral problem (the truncated Gelfand-Levitan potential,  \cite{Pryce})
\begin{gather*}
-u''+2\frac{T(x)\sin 2x+\cos^4 x}{T^2} u=\lambda u, \quad T(x)=1+\frac x2+\frac{\sin(2x)}4,\quad 0\le x\le 100,\\
u(0,\lambda)-u'(0,\lambda) = u(100,\lambda)=0.
\end{gather*}
The problem is considered difficult due to nonuniform oscillations of decreasing size in $q$.

\begin{figure}[htb!]
\centering
\includegraphics[bb=198 309 414 482, width=3in,height=2.4in]
{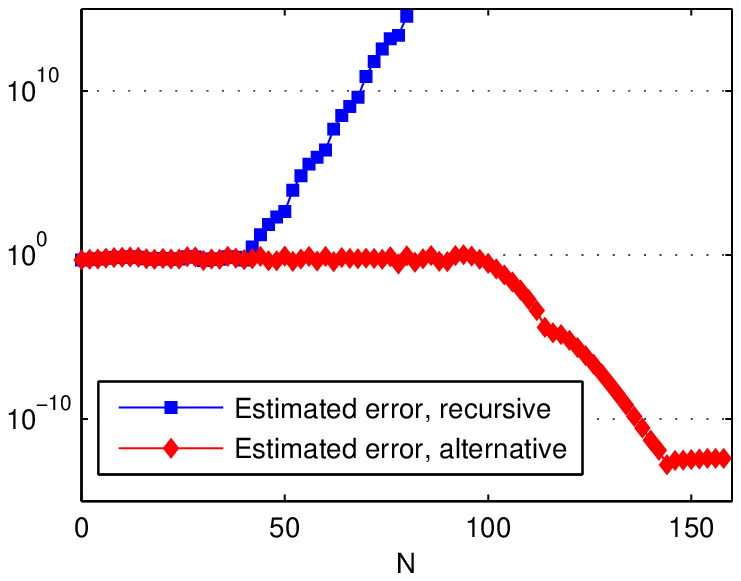}\ \ \
\includegraphics[bb=198 309 414 482, width=3in,height=2.4in]
{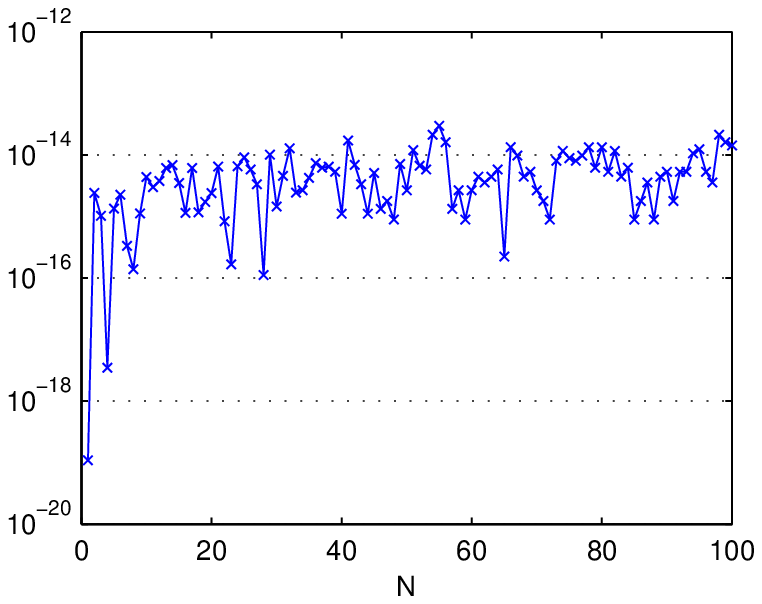}
\caption{Left figure: The plot of the maximum of the differences \eqref{ErrorCheck} at the point
$x=100$ from Example \ref{Example GelfLev}. The coefficients $\beta_k$, $k\le 160$ were obtained using formulas from
Definition \ref{Def beta} (blue line with square marks) and formulas
\eqref{sigma n} (red line with diamond marks). Right figure: errors of the first 100 eigenvalues in comparison with the \textsc{Matslise} package.}
\label{Ex3Fig1}
\end{figure}

We considered this problem in machine precision arithmetics, using 80001 points to represent all the functions involved. The particular solution was computed using the SPPS representation and was used to calculate the coefficients $\beta_k$, $k\le 160$. On Figure \ref{Ex3Fig1} (left plot) we present the maximum of the differences \eqref{ErrorCheck} at $x=100$. As one can see,  the differences do not decrease until $N=100$ and once again stabilize at $N=144$. The formulas from Definition \ref{Def beta} do not serve to compute that large number of the coefficients in the machine precision, while formulas \eqref{sigma n} provide such a possibility. We used $N=144$ to compute the approximate eigenvalues. Since the exact eigenvalues of the problem are not known, we compared our results to the values reported in \cite{Pryce} as well as to those produced by the \textsc{Matslise} package \cite{Ledoux}. In Table \ref{Ex3Tab1} we present several eigenvalues obtained. On Figure \ref{Ex3Fig1} (right plot) we show the differences between the eigenvalues computed by our method and by the \textsc{Matslise} package.

\begin{table}[htb]
\centering
\begin{tabular}{cccc}
\hline
$n$ & $\lambda_n$ (\cite{Pryce}) & $\lambda_n$ (our method) & $\lambda_n$ (\textsc{Matslise})\\
\hline
0 & 0.00024681157 & 0.000246811787231069 & 0.000246811787231069 \\
1 &  & 0.00222130735092850 & 0.00222130735093092 \\
2 &  & 0.00617030527111158 & 0.00617030527111055 \\
5 &  & 0.0298644887478121 & 0.0298644887478144 \\
10 &  & 0.108847814083180 & 0.108847814083183 \\
20 &  & 0.414974806699760 & 0.414974806699766 \\
50 &  & 2.51650713279491 & 2.51650713279492 \\
99 & 9.77082852816 & 9.77082852802586 & 9.77082852802587 \\
\hline
\end{tabular}\caption{Eigenvalues from Example \ref{Example GelfLev}.}
\label{Ex3Tab1}
\end{table}
\end{example}

\begin{example}\label{Example Sech}
Consider the following problem with spectral parameter depending boundary conditions
\begin{equation*}
\begin{cases}
-u''+Q(x) u=\lambda u, \qquad x\in [0,2a],\\
u'(0)-\nu u(0) = 0,\\
u'(2a)+\nu u(2a)=0,
\end{cases}
\end{equation*}
where $\lambda = -\nu^2$ and $Q(x)=-m(m+1)\sech^2(x-a)$, $m\in\mathbb{N}$. This spectral problem arises in relation with quantum wells when the sech-squared potential $q(x)=-m(m+1)\sech^2 x$, $x\in(-\infty, \infty)$ is truncated, see \cite{CKOR} and \cite[Sect. 7.4]{KT AnalyticApprox} for details. As was mentioned in \cite[Sect. 7.4]{KT AnalyticApprox}, the physically meaningful region for the eigenvalues of a quantum well problem is $\lambda \in \bigl[\min_{x\in[0,2a]} Q(x),0\bigr)$ and in the particular case under consideration, the non-truncated problem possesses exactly $m$ eigenvalues given by $\lambda_n=-(m-n)^2$, $n=0,1,\ldots,m-1$. In the notations of this work we have $\omega =i\nu$, the region to look for the eigenvalues is $\nu\in \bigl(0,\sqrt{m(m+1)}\bigr)$ and the characteristic function \eqref{SLCharEq} can be written in the form
\begin{equation*}
    c'(\omega,2a)-\frac{(i\omega+h)}w s'(\omega,2a)-i\omega c(\omega,2a)+i(i\omega+h)s(\omega,2a).
\end{equation*}
For the numerical experiment we chose $a=8$ and $m=3$ and $5$. All calculations were performed in Matlab in the machine precision. 50001 points (for $m=3$) and 80001 points (for $m=5$) were used to represent all the functions involved, and the alternative formulas \eqref{sigma n} and \eqref{tau n} were used to compute the coefficients $\beta_k$ and $\gamma_k$. For $m=3$ the optimal $N$ was found to be $51$, while for $m=5$ the optimal $N$ was $63$. The computed eigenvalues are presented in Table \ref{Ex4Tab1}.

\begin{table}[htb!]
\centering
\begin{tabular}{ccc}
$m=3$ & \quad & $m=5$\\
\begin{tabular}[t]
{cccc}\hline
$n$ & Exact $\lambda_{n}$ & $\lambda_{n}$ (our method) & $\lambda_{n}$ (\cite{CKOR})
\\\hline
0 & $-9$ & $ -9.00000001319$ & $-8.999628656$\\
1 & $-4$ & $ -3.99999999103$ & $-3.999998053$\\
2 & $-1$ & $ -1.00000000089$ & $-0.999927816$\\
\hline
\end{tabular} & &
\begin{tabular}[t]
{ccc}\hline
$n$ & Exact $\lambda_{n}$ & $\lambda_{n}$ (our method)
\\\hline
0 & $-25$ & $ -25.00003450$ \\
1 & $-16$ & $ -15.99994465$ \\
2 & $-9$ & $  -9.000027727$ \\
3 & $-4$ & $  -3.999995313$ \\
4 & $-1$ & $  -1.000000082$ \\
\hline
\end{tabular}
\end{tabular}
\caption{Approximations of $\lambda_n$ of the potential $-m(m+1)\sech^2 x$ (Example \ref{Example Sech}) for $m=3$ and $m=5$.}
\label{Ex4Tab1}
\end{table}
\end{example}

\begin{example}\label{Example Boumenir}
Consider the
following spectral problem with a complex valued potential \cite{Boumenir2001},
\begin{equation*}
\begin{cases}
-u''+(1+i)x^2 u=\lambda u, \quad 0\le x\le \pi,\\
u(0,\lambda) = u(\pi,\lambda)=0.
\end{cases}
\end{equation*}

\begin{figure}[htb!]
\centering
\includegraphics[bb=126 316 486 475, width=5in,height=2.2in]
{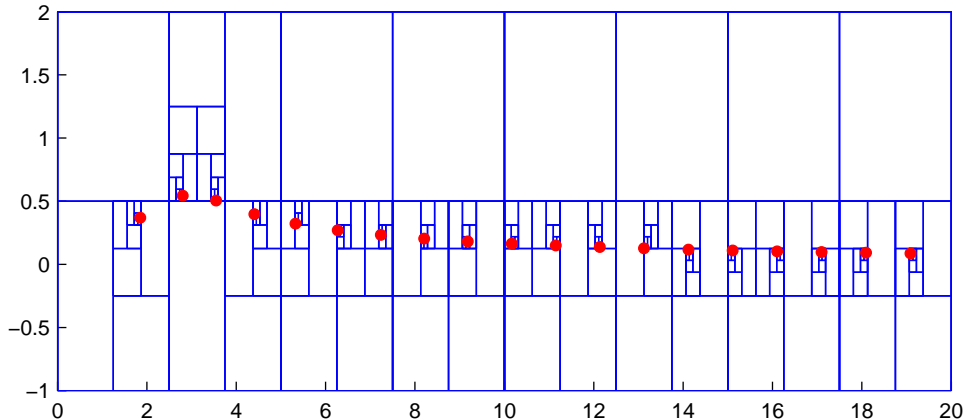}
\caption{Illustration to the work of the algorithm based on the argument principle in Example \ref{Example Boumenir}. Blue rectangles show the regions used to count the number of zeros on the subdivision step. Red circles mark the found eigenvalues $\omega_n=\sqrt{\lambda_n}$.}
\label{Ex5Fig1}
\end{figure}
We computed 40 eigenvalues of this problem in Matlab in the machine precision. All the functions involved were represented by their values in 4001 points. Alternative formulas for the coefficients $\beta_k$ were used and $N=27$ was find to be optimal for the machine precision arithmetics. The eigenvalues of the problem are complex numbers and they were localized using the argument principle, see \cite[Example 5.6]{KMoTo} for further details.
After the zeros were localized within rectangles with the sides smaller than 0.1, we applied several Newton iterations to obtain approximate eigenvalues. On Figure \ref{Ex5Fig1} we illustrate the work of the algorithm based on the argument principle. The approximate eigenvalues were compared with the values obtained from the exact characteristic equation using Wolfram Mathematica. We present several eigenvalues in Table \ref{Ex5Tab1}. The relative errors of the obtained eigenvalues were less than $9.5\cdot 10^{-15}$, while the relative errors of the eigenvalues reported in \cite{Boumenir2001} were between $4.9\cdot 10^{-6}$ and $5.8\cdot 10^{-4}$. Moreover, the absolute error of the $40$th eigenvalue was $2\times 10^{-12}$ (our method)
compared to $0.94$ (\cite{Boumenir2001}).

\begin{table}[htb!]
\centering
\small
\begin{tabular}{cccc}
\hline
$n$ & $\lambda_n$ (Exact) & $\lambda_n$ (our method) & $\lambda_n$ (\cite{Boumenir2001})\\
\hline
1 & $3.29252447095779+1.36633744750457i$ & $3.29252447095781 +      1.36633744750457i$ & $3.292530+1.366321i$ \\
2 & $7.55904717588980+3.05068659781596i$ & $7.55904717588983 +      3.05068659781599i$ & $7.559344+3.050506i$, \\
3 & $12.33985666951932+3.59139757785523i$ & $ 12.33985666951938 +      3.59139757785521i$ & $12.34084+3.59194i$, \\
5 & $28.26784723460268+3.43290953376002i$ & $28.26784723460276 +      3.43290953376014i$ & $28.26883+3.43560i$ \\
10 &$103.2845071723909+3.3276829117743i$ & $103.2845071723903 +      3.3276829117725i$ & $103.2855+3.3390i$ \\
20 &$403.2885933262633+3.2994144085675i$ & $403.2885933262637 +      3.2994144085679i$ & $403.2893+3.3512i$ \\
40 &$1603.289554053531+3.292259803191i$ & $1603.289554053531 +      3.292259803189i$ & $1603.139+4.214i$ \\
\hline
\end{tabular}\caption{Eigenvalues from Example \ref{Example Boumenir}.}
\label{Ex5Tab1}
\end{table}
\end{example}

\section{Conclusions}
The representations \eqref{c(omega,x) via bessel}, \eqref{s(omega,x) via bessel} of solutions to \eqref{SL_omega} are proved together with the
representations \eqref{c prime (omega)}, \eqref{s prime (omega)} of their derivatives. For the coefficients $\beta_{k}$ and
$\gamma_{k}$ in \eqref{c(omega,x) via bessel}, \eqref{s(omega,x) via bessel}, \eqref{c prime (omega)} and \eqref{s prime (omega)}  besides the closed form formulas \eqref{beta direct definition}, \eqref{gamma n} a recurrent integration procedure is developed. Estimates, uniform with respect to $\omega$, for the rate of convergence of the series involved in the representations are
obtained. It is shown that besides offering new analytical representations of
solutions, formulas \eqref{c(omega,x) via bessel}, \eqref{s(omega,x) via bessel}, \eqref{c prime (omega)} and \eqref{s prime (omega)} put at one's disposal a simple and powerful numerical
method for solving initial value and spectral problems for \eqref{SL_omega}. Due to the uniformity of the approximation with respect to $\omega$, the numerical method
based on \eqref{c(omega,x) via bessel}, \eqref{s(omega,x) via bessel}, \eqref{c prime (omega)} and \eqref{s prime (omega)} allows one to compute within seconds large sets of eigendata with
a nondeteriorating accuracy.

\end{document}